\documentclass[reqno,10pt]{amsart}
\usepackage{amscd,amssymb,verbatim,array}
\usepackage{hyperref}
\usepackage{amsmath}
\usepackage[active]{srcltx} 

\numberwithin{equation}{section} \theoremstyle{plain}

\newtheorem{theorem}{Theorem}[section]
\newtheorem{lemma}[theorem]{Lemma}

\newtheorem{corollary}[theorem]{Corollary}
\newtheorem{definition}[theorem]{Definition}

\theoremstyle{definition}

\theoremstyle{remark}

\numberwithin{equation}{section}

\newcommand{\Det}{\operatorname{Det}}
\newcommand{\Dir}{\operatorname{D}}

\newcommand{\Oo}{\operatorname{O}}
\newcommand{\Abs}{\operatorname{abs}}
\newcommand{\Rel}{\operatorname{rel}}

\newcommand{\Tan}{\operatorname{tan}}
\newcommand{\Nor}{\operatorname{nor}}
\newcommand{\Dim}{\operatorname{dim}}
\newcommand{\Ord}{\operatorname{ord}}

\newcommand{\Ric}{\operatorname{Ric}}
\newcommand{\Ker}{\operatorname{ker}}
\newcommand{\Spec}{\operatorname{Spec}}
\newcommand{\Tr}{\operatorname{Tr}}
\newcommand{\re}{\operatorname{Re}}

\newcommand{\I}{\operatorname{I}}
\newcommand{\II}{\operatorname{II}}
\newcommand{\III}{\operatorname{III}}
\newcommand{\V}{\operatorname{V}}
\newcommand{\IV}{\operatorname{IV}}

\newcommand{\ddet}{\operatorname{det}}
\newcommand{\Id}{\operatorname{Id}}

\newcommand{\vol}{\operatorname{vol}}

\begin{document}

\title[zeta-determinant of the Dirichlet-to-Neumann operator on forms]
{The zeta-determinant of the Dirichlet-to-Neumann operator of the Steklov Problem on forms}

\author{Klaus Kirsten}
\address{Mathematical Reviews, American Mathematical Society,
 416 4$th$ Street, Ann Arbor, MI 48103, USA}

\email{Klaus\_Kirsten@Baylor.edu}

\author{Yoonweon Lee}

\address{Department of Mathematics Education, Inha University, Incheon, 22212, Korea and
School of Mathematics, Korea Institute for Advanced Study, 85 Hoegiro, Dongdaemun-gu, Seoul, 02455, Korea
}

\email{yoonweon@inha.ac.kr}

\subjclass[2000]{Primary: 58J20; Secondary: 14F40}
\keywords{zeta-determinants of elliptic operators, Dirichlet-to-Neumann operator, curvature tensors, absolute/relative boundary conditions, conformal rescaling}

\begin{abstract}
On a compact Riemannian manifold $M$ with boundary $Y$, we express the log of the zeta-determinant of the Dirichlet-to-Neumann operator acting on $q$-forms on $Y$
as the difference of the log of the zeta-determinant of the Laplacian on $q$-forms on $M$ with absolute boundary conditions and
that of the Laplacian with Dirichlet boundary conditions with some additional terms which are expressed by curvature tensors.
When the dimension of $M$ is $2$ or $3$, we compute these terms explicitly.
We also discuss the value of the zeta function at zero associated to the Dirichlet-to-Neumann operator by using a conformal rescaling method.
As an application, we recover the result of the conformal invariance obtained in \cite{GG} when $\Dim M = 2$.
\end{abstract}

\maketitle

\section{Introduction}

Let $(M, Y;g)$ be an $m$-dimensional Riemannian manifold with smooth boundary $Y$ and $\Omega^{q}(M)$ be the space of smooth $q$-forms.
We consider the exterior derivative $d_{q} : \Omega^{q}(M) \rightarrow \Omega^{q+1}(M)$ and its formal adjoint $\delta_{q} = (-1)^{mq+1} \star_{M} d \star_{M}$, where $\star_{M}$ is the Hodge star operator.
Then the Hodge-De Rham Laplacian $\Delta^{q}_{M}$ acting on $\Omega^{q}(M)$ is defined by
$\Delta_{M}^{q} = \delta_{q} d_{q} + d_{q-1} \delta_{q-1}$.
If there is no confusion, we will drop the $q$ on $d_{q}$ and $\delta_{q}$.
We choose a collar neighborhood $U$ of $Y$ which is diffeomorphic to $Y \times [0, 1)$ and denote the canonical inclusion by $i : Y \rightarrow M$. We also choose a unit vector field $\frac{\partial}{\partial {u}}$ which is an  inward normal vector to $Y$ and denote the dual by $du$.
We write a $q$-form $\omega$ on $U$ by $\omega = \omega_{1} + du \wedge \omega_{2}$, and define the tangential part $\omega_{\Tan}$ and normal part $\omega_{\Nor}$ of $\omega$ as follows.

\begin{eqnarray}     \label{E:1.1}
\omega_{\Tan} := i^{\ast} \omega = \omega_{1}|_{Y}, \qquad \omega_{\Nor} = i^{\ast} \left( \iota_{\frac{\partial}{\partial_{u}}} \omega \right) = \omega_{2}|_{Y},
\end{eqnarray}

\noindent
where $\iota_{\frac{\partial}{\partial {u}}} \omega$ is the interior product of $\omega$ and $\frac{\partial}{\partial {u}}$.
A $q$-form $\omega$ is said to satisfy absolute boundary conditions if $\omega_{\Nor} = (d \omega)_{\Nor} = 0$,
and it is said to satisfy relative boundary conditions if $\omega_{\Tan} = (\delta \omega)_{\Tan} = 0$.
We denote by $\Omega^{q}_{\Abs / \Rel}(M)$ the space of smooth $q$-forms satisfying absolute/relative boundary conditions, {\it i.e.}

\begin{eqnarray}     \label{E:1.2}
\Omega^{q}_{\Abs}(M)  =  \{ \omega \in \Omega^{q}(M) \mid \omega_{\Nor} = (d \omega)_{\Nor} = 0 \}, \quad
\Omega^{q}_{\Rel}(M)  =  \{ \omega \in \Omega^{q}(M) \mid \omega_{\Tan} = (\delta \omega)_{\Tan} = 0 \}.
\end{eqnarray}

\noindent
We also denote by $\Delta^{q}_{M, \Abs/\Rel}$ and $\Delta^{q}_{M, \Dir}$ the Laplacian $\Delta^{q}_{M}$ with absolute/relative and Dirichlet boundary conditions, respectively.
Then, $\Delta^{q}_{M, \Abs/\Rel}$ and $\Delta^{q}_{M, \Dir}$ are self-adjoint operators having discrete eigenvalues.
We note that for $q = 0$, absolute/relative boundary conditions are equal to Neumann/Dirichlet boundary conditions.
For $0 \leq \lambda \in {\mathbb R}$, we define the Dirichlet-to-Neumann operator $Q^{q}_{\Abs}(\lambda)$ and
$Q^{q}_{\Rel}(\lambda)$ acting on $\Omega^{q}(Y)$ as in \cite{Ca, RS, Ta}.
For $\varphi \in \Omega^{q}(Y)$, we choose arbitrary extensions $\phi \in \Omega^{q}(M)$ and
${\widetilde \phi} \in \Omega^{q+1}(M)$ of $\varphi$ and $du \wedge \varphi$ satisfying

\begin{eqnarray}  \label{E:1.3}
i^{\ast} \phi = \varphi,  \qquad i^{\ast} \left( \iota_{\frac{\partial}{\partial_{u}}} \phi \right) = 0, \qquad
i^{\ast} {\widetilde \phi} = 0, \qquad
i^{\ast} \left( \iota_{\frac{\partial}{\partial_{u}}} {\widetilde \phi} \right) = \varphi.
\end{eqnarray}

\noindent
We define the Poisson operators

\begin{eqnarray}    \label{E:1.4}
& & {\mathcal P}^{q}_{\Abs}(\lambda) : \Omega^{q}(Y) \rightarrow \Omega^{q}(M), \qquad
{\mathcal P}^{q}_{\Abs} \varphi := \phi -
\left( \Delta^{q}_{M, \Dir} + \lambda \right)^{-1} \left( \Delta^{q}_{M} + \lambda \right) \phi,  \\
& & {\mathcal P}^{q}_{\Rel}(\lambda) : \Omega^{q}(Y) \rightarrow \Omega^{q+1}(M), \qquad
{\mathcal P}^{q}_{\Rel} \varphi := {\widetilde \phi} -
\left( \Delta^{q+1}_{M, \Dir} + \lambda \right)^{-1} \left( \Delta^{q+1}_{M} + \lambda \right) {\widetilde \phi}.  \nonumber
\end{eqnarray}

\noindent
It is not difficult to see that the definition of ${\mathcal P}^{q}_{\Abs}(\lambda)$ and
${\mathcal P}^{q}_{\Rel}(\lambda)$ do not depend on the choices of the extensions of $\phi$ and ${\widetilde \phi}$ \cite{Ca,MM}.  ${\mathcal P}^{q}_{\Abs}(\lambda)$ and
${\mathcal P}^{q}_{\Rel}(\lambda)$ satisfy the following relations.

\begin{eqnarray}    \label{E:1.5}
& & \left( \Delta^{q}_{M} + \lambda \right) {\mathcal P}^{q}_{\Abs}(\lambda) \varphi = 0, \qquad
i^{\ast} {\mathcal P}^{q}_{\Abs}(\lambda) \varphi = \varphi, \qquad
i^{\ast} \left( \iota_{\frac{\partial}{\partial_{u}}} {\mathcal P}^{q}_{\Abs}(\lambda) \varphi \right) = 0,  \\
& & \left( \Delta^{q+1}_{M} + \lambda \right) {\mathcal P}^{q}_{\Rel}(\lambda) \varphi = 0, \qquad
i^{\ast} {\mathcal P}^{q}_{\Rel}(\lambda) \varphi = 0, \qquad
i^{\ast} \left( \iota_{\frac{\partial}{\partial_{u}}} {\mathcal P}^{q}_{\Rel}(\lambda) \varphi \right) = \varphi.
\nonumber
\end{eqnarray}

\begin{definition}  \label{Definition:1.1}
We define two Drichlet-to-Neumann operators $Q^{q}_{\Abs}(\lambda)$ and
$Q^{q}_{\Rel}(\lambda)$ as follows \cite{Ca, RS, Ta}.
\begin{eqnarray*}
& & Q^{q}_{\Abs}(\lambda) : \Omega^{q}(Y) \rightarrow \Omega^{q}(Y), \qquad
Q^{q}_{\Abs}(\lambda) (\varphi) = - i^{\ast} \left( \iota_{\frac{\partial}{\partial_{u}}} d
{\mathcal P}^{q}_{\Abs}(\lambda) \varphi \right),  \\
& & Q^{q}_{\Rel}(\lambda) : \Omega^{q}(Y) \rightarrow \Omega^{q}(Y), \qquad
Q^{q}_{\Rel}(\lambda) (\varphi) = i^{\ast} \left( \delta {\mathcal P}^{q}_{\Rel}(\lambda) \varphi \right).
\end{eqnarray*}
\end{definition}

\vspace{0.2 cm}
\noindent
{\it Remark} : (1) When $q=0$ and $\lambda = 0$, $Q^{0}_{\Abs}(0)$ is the usual Dirichlet-to-Neumann operator on
the Steklov problem on the space of smooth functions.

\noindent
(2) In eq.(\ref{E:3.12}) below, $Q^{q}_{\Abs}(\lambda)$ and $Q^{q-1}_{\Rel}(\lambda)$ are defined by using a local coordinate system, which is more intuitive.

\vspace{0.2 cm}

The Green formula for the Hodge-De Rham Laplacians is given as follows \cite{Mo,Sc}. For $\omega$, $\theta \in \Omega^{q}(M)$,

\begin{eqnarray}    \label{E:1.6}
\langle d \omega, d \theta \rangle_{M} + \langle \delta \omega, \delta \theta \rangle_{M}
& = & \langle \Delta^{q}_{M} \omega, \theta \rangle_{M} +
\int_{Y} i^{\ast}(\theta \wedge \star_{M} d \omega - \delta \omega \wedge \star_{M} \theta),
\end{eqnarray}

\noindent
where we use the convention that $d \vol(M) = - du \wedge d\vol(Y)$ on $Y$.
For $\varphi_{1}$, $\varphi_{2} \in \Omega^{q}(Y)$, eq.(\ref{E:1.6}) shows that

\begin{eqnarray*}
\langle Q^{q}_{\Abs}(\lambda) \varphi_{1}, ~  \varphi_{2}  \rangle_{Y} & = &
\lambda \langle {\mathcal P}^{q}_{\Abs} \varphi_{1}, ~ {\mathcal P}^{q}_{\Abs} \varphi_{2}  \rangle_{M} ~ + ~
\langle d {\mathcal P}^{q}_{\Abs} \varphi_{1}, ~ d {\mathcal P}^{q}_{\Abs} \varphi_{2}  \rangle_{M} ~ + ~
\langle \delta {\mathcal P}^{q}_{\Abs} \varphi_{1}, ~ \delta {\mathcal P}^{q}_{\Abs} \varphi_{2}  \rangle_{M}, \\
\langle Q^{q}_{\Rel}(\lambda) \varphi_{1}, ~ \varphi_{2}  \rangle_{Y} & = &
\lambda \langle {\mathcal P}^{q}_{\Rel} \varphi_{1}, ~ {\mathcal P}^{q}_{\Rel} \varphi_{2}  \rangle_{M} ~ + ~
\langle d {\mathcal P}^{q}_{\Rel} \varphi_{1}, ~ d {\mathcal P}^{q}_{\Rel} \varphi_{2}  \rangle_{M} ~ + ~
\langle \delta {\mathcal P}^{q}_{\Rel} \varphi_{1}, ~ \delta {\mathcal P}^{q}_{\Rel} \varphi_{2}  \rangle_{M},
\end{eqnarray*}

\noindent
which implies that $Q^{q}_{\Abs/\Rel}(\lambda)$ are non-negative self-adjoint operators. Moreover, they are elliptic $\Psi$DO's with parameter $\lambda$ of order $1$ and weight $2$ with the principal symbol $\sigma_{L}(Q^{q}_{\Abs/\Rel}(\lambda))(x, \xi) = \sqrt{|\xi|^{2} + \lambda}$ (see \cite{BFK,Sh} for a $\Psi$DO's with parameter).
It is well known (Theorem 2.7.3 in \cite{Gi1}) that

\begin{eqnarray}    \label{E:1.7}
& & \Ker \Delta_{M, \Abs}^{q} = \{ \omega \in \Omega^{q}_{\Abs}(M) \mid d \omega = \delta \omega = 0 \} \cong H^{q}(M), \\
& & \Ker \Delta_{M, \Rel}^{q} = \{ \omega \in \Omega^{q}_{\Rel}(M) \mid d \omega = \delta \omega = 0 \} \cong H^{q}(M, Y),   \nonumber
\end{eqnarray}

\noindent
which shows that

\begin{eqnarray}    \label{E:1.8}
\Ker Q^{q}_{\Abs}(0) = \left\{ i^{\ast} \omega \mid \omega \in \Ker \Delta_{M, \Abs}^{q} \right\}, \qquad
\Ker Q^{q}_{\Rel}(0) = \left\{ i^{\ast} \left( \iota_{\frac{\partial}{\partial_{u}}} \omega \right) \mid \omega \in \Ker \Delta_{M, \Rel}^{q+1} \right\}.
\end{eqnarray}

\noindent
In particular, it follows that

\begin{eqnarray}   \label{E:1.9}
\Dim \Ker Q^{q}_{\Abs}(0) & = & \Dim \Ker \Delta_{M, \Abs}^{q} = \Dim H^{q}(M), \\
\Dim \Ker Q^{q}_{\Rel}(0) & = & \Dim \Ker \Delta_{M, \Rel}^{q+1} = \Dim H^{q+1}(M, Y).  \nonumber
\end{eqnarray}

For ${\mathcal D} = \Delta^{q}_{M, \Abs/\Rel} + \lambda$ or $Q^{q}(\lambda)_{\Abs/\Rel}$, we define the zeta function $\zeta_{{\mathcal D}}(s)$ by

\begin{eqnarray}  \label{E:1.10}
\zeta_{{\mathcal D}}(s) & = & \frac{1}{\Gamma(s)} \int_{0}^{\infty} t^{s-1} \left( \Tr e^{- t {\mathcal D}} - \Dim \Ker {\mathcal D} \right) dt.
\end{eqnarray}

\noindent
It is well known that $\zeta_{{\mathcal D}}(s)$ is analytic for $\Re s > \frac{m}{\Ord({\mathcal D})}$ and has a meromorphic continuation to the whole complex plane having a regular value at $s=0$.
When $\Ker {\mathcal D} = \{ 0 \}$, we define the zeta-determinant of ${\mathcal D}$ by $\Det {\mathcal D} = e^{-\zeta_{{\mathcal D}}^{\prime}(0)}$.
If $\Dim \Ker {\mathcal D} \geq 1$, we define the modified zeta-determinant by the same formula, which we denote by $\Det^{\ast} {\mathcal D} = e^{- \zeta_{{\mathcal D}}^{\prime}(0)}$.
In this paper, we are going to discuss

\begin{eqnarray}    \label{E:1.11}
& & \ln \Det^{\ast} \Delta^{q}_{M, \Abs} - \ln \Det \Delta^{q}_{M, \Dir} - \ln \Det^{\ast} Q^{q}_{\Abs}(0)
\quad \text{and} \quad \\
& & \ln \Det^{\ast} \Delta^{q+1}_{M, \Rel} - \ln \Det \Delta^{q+1}_{M, \Dir} - \ln \Det^{\ast} Q^{q}_{\Rel}(0), \nonumber
\end{eqnarray}

\noindent
and their applications.
However, for the Hodge star operators $\star_{M}$ and $\star_{Y}$ of $M$ and $Y$, simple computation shows that

\begin{eqnarray}    \label{E:1.12}
\star_{M}^{-1} \Delta^{m-q}_{M, \Rel} \star_{M} = \Delta^{q}_{M, \Abs}, \qquad \star_{M}^{-1} \Delta^{m-q}_{M, \Dir} \star_{M} = \Delta^{q}_{M, \Dir},
\quad \text{and} \quad \star_{Y}^{-1} Q^{m-1-q}_{\Rel}(\lambda) \star_{Y} = Q^{q}_{\Abs}(\lambda),
\end{eqnarray}

\noindent
which shows that

\begin{eqnarray}    \label{E:1.13}
& & \ln \Det^{\ast} \Delta^{q}_{M, \Abs} - \ln \Det \Delta^{q}_{M, \Dir} - \ln \Det^{\ast} Q^{q}_{\Abs}(0)  \\
& = & \ln \Det^{\ast} \Delta^{m-q}_{M, \Rel} - \ln \Det \Delta^{m-q}_{M, \Dir} - \ln \Det^{\ast} Q^{m-1-q}_{\Rel}(0).  \nonumber
\end{eqnarray}

\noindent
Hence, it is enough to consider $\ln \Det^{\ast} \Delta^{q}_{M, \Abs} - \ln \Det \Delta^{q}_{M, \Dir} - \ln \Det^{\ast} Q^{q}_{\Abs}(0)$.

In this paper, we use the method of proving the BFK-gluing formula for zeta-determinants
of Laplacians \cite{BFK,Ca,Fo} to show that $~~\ln \Det^{\ast} \Delta^{q}_{M, \Abs} - \ln \Det \Delta^{q}_{M, \Dir} - \ln \Det^{\ast} Q^{q}_{\Abs}(0)~~$ is expressed by some curvature tensors on $Y$. We compute it explicitly when $\Dim M = 2$ and $3$.
We also discuss the value of the zeta function $\zeta_{Q^{q}_{\Abs}}(s)$ at $s=0$
by using the conformal metric rescaling method.
Finally, when $M$ is a $2$-dimensional smooth Riemannian manifold with smooth boundary $Y$ and $\ell(Y)$ is the length of $Y$,
we show that $~\frac{1}{\ell(Y)} \Det^{\ast} Q^{0}_{\Abs}(0)~$ is a conformal invariant,
which was proved earlier by Guillarmou and Guillop\'e in \cite{GG} (see also \cite{EW}).

\vspace{0.3 cm}

\section{Relation between $\ln \Det^{\ast} \Delta^{q}_{M, \Abs} - \ln \Det \Delta^{q}_{M, \Dir}$ and $\ln \Det^{\ast} Q^{q}_{\Abs}(0)$}

In this section, we are going to discuss the relation between $~ \ln \Det^{\ast} \Delta^{q}_{M, \Abs} - \ln \Det \Delta^{q}_{M, \Dir} ~$ and
$~ \ln \Det^{\ast} Q^{q}_{\Abs}(0) ~$. We first recall that for $\lambda \neq 0$,

\begin{eqnarray}    \label{E:2.1}
& & ( \Delta^{q}_{M} + \lambda ) \cdot {\mathcal P}^{q}_{\Abs}(\lambda) = 0, \qquad i^{\ast} {\mathcal P}^{q}_{\Abs}(\lambda) = \Id, \qquad
i^{\ast} \iota_{\frac{\partial}{\partial u}}{\mathcal P}^{q}_{\Abs}(\lambda) = 0,
\end{eqnarray}

\noindent
which shows that

\begin{eqnarray}   \label{E:2.2}
& & \left( \Delta^{q}_{M, \Abs} + \lambda \right)^{-1} - \left( \Delta^{q}_{M, \Dir} + \lambda \right)^{-1} ~ = ~
{\mathcal P}^{q}_{\Abs}(\lambda) i^{\ast} \left( \Delta^{q}_{M, \Abs} + \lambda \right)^{-1}.
\end{eqnarray}

\begin{lemma}  \label{Lemma:2.1}
For $\lambda \neq 0$ we have
\begin{eqnarray*}
\frac{d}{d \lambda} {\mathcal P}^{q}_{\Abs}(\lambda) & = & - \left( \Delta^{q}_{M, \Dir} + \lambda \right)^{-1} {\mathcal P}^{q}_{\Abs}(\lambda).
\end{eqnarray*}
\end{lemma}

\begin{proof}
Taking the derivative of (\ref{E:2.1}), we obtain the following equalities.
\begin{eqnarray*}
{\mathcal P}^{q}_{\Abs}(\lambda) + ( \Delta^{q}_{M} + \lambda ) \cdot \frac{d}{d \lambda} {\mathcal P}^{q}_{\Abs}(\lambda) = 0, \qquad i^{\ast} \frac{d}{d \lambda} {\mathcal P}^{q}_{\Abs}(\lambda) = 0, \qquad
i^{\ast} \iota_{\frac{\partial}{\partial u}} \frac{d}{d \lambda} {\mathcal P}^{q}_{\Abs}(\lambda) = 0,
\end{eqnarray*}
which yields the conclusion.
\end{proof}

\noindent
From Definition \ref{Definition:1.1} we note that

\begin{eqnarray}    \label{E:2.3}
\frac{d}{d \lambda} Q_{\Abs}^{q}(\lambda) & = &
- i^{\ast} \iota_{\frac{\partial}{\partial_{u}}} ~ d ~ \frac{d}{d\lambda} {\mathcal P}_{\Abs}^{q}(\lambda)
~ = ~ - i^{\ast} \iota_{\frac{\partial}{\partial_{u}}} ~ d ~ \left( - \left( \Delta^{q}_{M, \Dir} + \lambda \right)^{-1} \right) {\mathcal P}^{q}_{\Abs}(\lambda)   \\
& = & - i^{\ast} \iota_{\frac{\partial}{\partial_{u}}} ~ d ~ \left( \left( \Delta^{q}_{M, \Abs} + \lambda \right)^{-1} - \left( \Delta^{q}_{M, \Dir} + \lambda \right)^{-1} \right) {\mathcal P}^{q}_{\Abs}(\lambda)     \nonumber \\
& = & - i^{\ast} \iota_{\frac{\partial}{\partial_{u}}} ~ d ~ {\mathcal P}^{q}_{\Abs}(\lambda) i^{\ast} \left( \Delta^{q}_{M, \Abs} + \lambda \right)^{-1} {\mathcal P}^{q}_{\Abs}(\lambda)    \nonumber \\
& = & Q_{\Abs}^{q}(\lambda) i^{\ast} \left( \Delta^{q}_{M, \Abs} + \lambda \right)^{-1} {\mathcal P}^{q}_{\Abs}(\lambda),   \nonumber
\end{eqnarray}

\noindent
where in the third equality we used the fact that $\omega_{\Nor} = (d \omega)_{\Nor} = 0$ for $\omega \in \Omega^{q}_{\Abs}(M)$.
This yields

\begin{eqnarray}    \label{E:2.4}
Q_{\Abs}^{q}(\lambda)^{-1} \frac{d}{d \lambda} Q_{\Abs}^{q}(\lambda) & = &
i^{\ast} \left( \Delta^{q}_{M, \Abs} + \lambda \right)^{-1} {\mathcal P}^{q}_{\Abs}(\lambda).
\end{eqnarray}

\noindent
For $\nu = [\frac{m-1}{2}] + 1$, we also note that

\begin{eqnarray}    \label{E:2.5}
& & \frac{d^{\nu}}{d\lambda^{\nu}} \left\{ \log \Det \left( \Delta_{M, \Abs}^{q} + \lambda \right) -
\log \Det \left( \Delta_{M, \Dir}^{q} + \lambda \right)   \right\}   \\
& = &
\Tr \left\{ \frac{d^{\nu - 1}}{d\lambda^{\nu - 1}} \left( \left( \Delta_{M, \Abs}^{q} + \lambda \right)^{-1} -
\left( \Delta_{M, \Dir}^{q} + \lambda \right)^{-1} \right) \right\}   \nonumber   \\
& = & \Tr \left\{ \frac{d^{\nu - 1}}{d\lambda^{\nu - 1}}
\left( {\mathcal P}_{\Abs}^q(\lambda) i^{\ast} \left( \Delta_{M, \Abs}^{q} + \lambda \right)^{-1} \right) \right\}   \nonumber  \\
& = & \Tr \left\{ \frac{d^{\nu - 1}}{d\lambda^{\nu - 1}}
\left( i^{\ast} \left( \Delta_{M, \Abs}^{q} + \lambda \right)^{-1} {\mathcal P}_{\Abs}^q(\lambda) \right) \right\}
~ = ~ \Tr \left\{ \frac{d^{\nu - 1}}{d\lambda^{\nu - 1}}
\left( Q_{\Abs}^{q}(\lambda)^{-1} \frac{d}{d \lambda} Q_{\Abs}^{q}(\lambda) \right) \right\}    \nonumber  \\
& = & \frac{d^{\nu}}{d\lambda^{\nu}} \log \Det Q_{\Abs}^{q}(\lambda).   \nonumber
\end{eqnarray}

\noindent
This equality leads to the following result, which has been already proved in Theorem 6.2 of \cite{Ca} by a different method.

\begin{lemma}  \label{Lemma:2.2}
There exists a polynomial $P^{q}(\lambda) = \sum_{k=0}^{[\frac{m-1}{2}]} a_{k} \lambda^{k}$ such that
\begin{eqnarray*}
& & \log \Det \left( \Delta_{M, \Abs}^{q} + \lambda \right) -
\log \Det \left( \Delta_{M, \Dir}^{q} + \lambda \right) ~ = ~ \sum_{k=0}^{[\frac{m-1}{2}]} a_{k} \lambda^{k} + \ln \Det Q^{q}_{\Abs}(\lambda).
\end{eqnarray*}
\end{lemma}

To determine the coefficients $a_{k}$, we are going to consider the asymptotic expansion of each term in Lemma \ref{Lemma:2.2} for $\lambda \rightarrow \infty$. When $t \rightarrow 0^{+}$, it is well known \cite{Gi1} that for some ${\frak a}_{j}$, ${\frak b}_{j} \in {\mathbb R}$,

\begin{eqnarray}   \label{E:2.6}
\Tr e^{- t \Delta^{q}_{M, \Abs}} & \sim & \sum_{j=0}^{\infty} {\frak a}_{j} t^{- \frac{m-j}{2}}, \qquad
\Tr e^{- t \Delta^{q}_{M, \Dir}} ~ \sim ~ \sum_{j=0}^{\infty} {\frak b}_{j} t^{- \frac{m-j}{2}} .
\end{eqnarray}

\noindent
It is straightforward ((5.1) in \cite{Vo}, Lemma 2.1 in \cite{KL2}) that for $\lambda \rightarrow \infty$,

\begin{eqnarray}    \label{E:2.7}
& & \ln \Det \left( \Delta^{q}_{M, \Abs} + \lambda \right) - \ln \Det \left( \Delta^{q}_{M, \Dir} + \lambda \right) ~ \sim ~
- \sum_{\stackrel{j=0}{j\neq m}}^{N} ({\frak a}_{j} - {\frak b}_{j}) \frac{d}{ds} \left.\left( \frac{\Gamma(s - \frac{m-j}{2})}{\Gamma(s)} \right)\right|_{s=0} \lambda^{\frac{m-j}{2}}    \\
& & \hspace{0.5 cm}  + ~ ({\frak a}_{m} - {\frak b}_{m}) \ln \lambda ~ + ~ \sum_{j=0}^{m-1} ({\frak a}_{j} - {\frak b}_{j})\left.
\left( \frac{\Gamma(s - \frac{m-j}{2})}{\Gamma(s)} \right)\right|_{s=0} \lambda^{\frac{m-j}{2}} \ln \lambda ~ + ~ O(\lambda^{-\frac{N + 1 - m}{2}}),   \nonumber
\end{eqnarray}

\noindent
where we note that the constant term does not appear.
Since $Q^{q}_{\Abs}(\lambda)$ is an elliptic $\Psi$DO of order 1 with parameter of weight $2$,
it is shown in the Appendix of \cite{BFK} that
for $\lambda \rightarrow \infty$,
$\ln \Det Q^{q}_{\Abs}(\lambda)$ has the following asymptotic expansion,

\begin{eqnarray}   \label{E:2.8}
\ln \Det Q^{q}_{\Abs}(\lambda) & \sim & \sum_{j=0}^{\infty} \pi_{j} \lambda^{\frac{m-1-j}{2}} + \sum_{j=0}^{m-1} q_{j} \lambda^{\frac{m-1-j}{2}} \ln \lambda,
\end{eqnarray}

\noindent
where $\pi_{j}$ and $q_{j}$ are locally computable as follows. For a fixed local coordinate system we denote the homogeneous symbols of $Q^{q}_{\Abs}(\lambda)$
and its resolvent $(\mu - Q^{q}_{\Abs}(\lambda))^{-1}$ by

\begin{eqnarray}    \label{E:2.9}
& & \sigma(Q^{q}_{\Abs}(\lambda))(y, \xi, \lambda) ~ \sim ~ \sum_{j=0}^{\infty} \widetilde{\alpha}_{1-j} (y, \xi, \lambda),  \\
& & \sigma\left(\left(\mu - Q^{q}_{\Abs}(\lambda)\right)^{-1}\right)(y, \xi, \lambda, \mu) ~ \sim ~ \sum_{j=0}^{\infty} \widetilde{r}_{-1-j} (y, \xi, \lambda, \mu).   \nonumber
\end{eqnarray}

\noindent
The densities $\pi_{j}(y)$ and $q_{j}(y)$ are computed as follows (Appendix of \cite{BFK}).

\begin{eqnarray}    \label{E:2.10}
& & \pi_{j}(y) ~ = ~ - \frac{\partial}{\partial s}\bigg|_{s=0} \left( \frac{1}{(2 \pi)^{m-1}} \int_{T_{y}^{\ast}Y}
\frac{1}{2 \pi i} \int_{\gamma} \mu^{-s} \Tr \widetilde{r}_{-1-j} \left(y, \xi, \frac{\lambda}{|\lambda|}, \mu\right) d\mu d \xi \right),       \\
& & q_{j}(y) ~ = ~ \frac{1}{2} ~ \left( \frac{1}{(2 \pi)^{m-1}} \int_{T_{y}^{\ast}Y}
\frac{1}{2 \pi i} \int_{\gamma} \mu^{-s} \Tr \widetilde{r}_{-1-j} \left(y, \xi, \frac{\lambda}{|\lambda|}, \mu\right) d\mu d \xi \right)\bigg|_{s=0}, \nonumber  \\
& & \pi_{j} ~ = ~ \int_{Y} \pi_{j}(y) ~ dy,
\qquad q_{j} ~ = ~ \int_{Y} q_{j}(y) ~ dy,   \nonumber
\end{eqnarray}

\noindent
where $dy$ is the volume form on $Y$ and $\gamma$ is a contour enclosing the poles of $\Tr \widetilde{r}_{-1-j} \left(y, \xi, \frac{\lambda}{|\lambda|}, \mu\right)$ counterclockwise.
Comparing the coefficients of $\lambda^{j}$, we have the following result.

\begin{lemma}  \label{Lemma:2.3}
\begin{eqnarray*}
& & a_{0} ~ = ~ - \pi_{m-1}, \qquad a_{k} ~ = ~ - \pi_{m-1-2k} - ({\frak a}_{m-2k} - {\frak b}_{m - 2k}) \left( \frac{d}{ds} \frac{\Gamma(s-k)}{\Gamma(s)} \right)\bigg|_{s=0}, \quad 1 \leq k \leq [(m-1)/2] , \\
& & q_{m-1} ~ = ~ {\frak a}_{m} - {\frak b}_{m}, \qquad q_{k} ~ = ~ ({\frak a}_{k+1} - {\frak b}_{k+1}) \left( \frac{\Gamma(s - \frac{m-k-1}{2})}{\Gamma(s)}   \right)\bigg|_{s=0}, \quad 0 \leq k \leq m-2 .
\end{eqnarray*}
\end{lemma}

\vspace{0.2 cm}
\noindent
Since the heat coefficients are quite well known \cite{Gi3,Ki}, we are going to concentrate on computing the $\pi_{k}$'s to determine the coefficients $a_{k}$.
We note that the coefficients $\pi_{k}$ are expressed by some curvature tensors including the scalar curvatures and principal curvatures of $Y$ in $M$ like heat coefficients (cf.\cite{PS}).
In Section 3 we are going to compute $\pi_{1}$ and $\pi_{2}$ along these lines when $\Dim M = 2, ~ 3$.

Before going further, we make one observation.
If $M$ has a product structure near $Y$ so that $\Delta^{q}_{M}$ is $- \partial_{y_{m}}^{2} + \Delta_{Y}$ on a collar neighborhood of $Y$, it is known that
$~Q^{q}_{\Abs}(\lambda) = \sqrt{\Delta_{Y} + \lambda} ~ + ~ a ~ smoothing ~ operator $ (cf. \cite{Le2,PW}).
In this case, $\ln \Det Q^{q}_{\Abs}(\lambda)$ and $\ln \Det \sqrt{\Delta_{Y} + \lambda}$ have the same asymptotic expansions for $\lambda \rightarrow \infty$, which is shown in the Appendix of \cite{BFK}.
Since $\ln \Det \sqrt{\Delta_{Y} + \lambda} = \frac{1}{2} \ln \Det (\Delta_{Y} + \lambda)$, the constant term in the asymptotic expansion of $\ln \Det Q^{q}_{\Abs}(\lambda)$ is zero (cf.(\ref{E:2.8})),
which shows that $a_{0} = 0$.

We next discuss the asymptotic behavior of each term in Lemma \ref{Lemma:2.2} for $\lambda \rightarrow 0$.
We first note that

\begin{eqnarray}   \label{E:2.11}
\ln \Det (\Delta^{q}_{M, \Dir} + \lambda) & = & \ln \Det \Delta^{q}_{M, \Dir} + O(\lambda).
\end{eqnarray}

\noindent
In view of (\ref{E:1.9}), we let $\Dim \Ker \Delta^{q}_{M, \Abs} = \Dim \Ker Q^{q}_{\Abs}(0) = \ell_{q}$ and
$\{ \psi_{1}(0), \cdots, \psi_{\ell_{q}}(0) \}$ be an orthonormal basis for $ \Ker \Delta^{q}_{M, \Abs}$.
Considering $\lambda \in {\Bbb C}-(-\infty , 0)$, $Q^{q}_{\Abs}(\lambda)$ is a self-adjoint holomorphic family of type (A) in the sense of T. Kato (for the definition see p. 375 of \cite{Ka})
and Theorem 3.9 on p. 392 of \cite{Ka} shows that there exist holomorphic families $\{ \theta_{j}(\lambda) \mid j = 1, 2, \cdots \}$ and $\{ \phi_{j}(\lambda) \mid j = 1, 2, \cdots \}$ of
eigenvalues and corresponding orthonormal eigensections of $Q^{q}_{\Abs}(\lambda)$ such that $0 < \theta_{1}(\lambda) \leq \cdots \leq \theta_{\ell_{q}}(\lambda) < \theta_{\ell_{q}+1}(\lambda) \leq \cdots $ and

\begin{eqnarray}     \label{E:2.12}
\parallel \phi_{j}(\lambda) \parallel = 1, \qquad \lim_{\lambda \rightarrow 0} \theta_{j}(\lambda) = 0 \quad \text{for} \quad 1 \leq j \leq \ell_{q}.
\end{eqnarray}

\noindent
This leads to

\begin{eqnarray}    \label{E:2.13}
& & \ln \Det (\Delta^{q}_{M, \Abs} + \lambda) ~ = ~ \ell_{q} \ln \lambda ~ + ~ \ln \Det^{\ast} \Delta^{q}_{M, \Abs} ~ + ~ O(\lambda) ,  \\
& & \ln \Det Q^{q}_{\Abs}(\lambda) ~ = ~ \ln \theta_{1}(\lambda) \cdots \theta_{\ell_{q}}(\lambda) ~ + ~ \ln \Det^{\ast} Q^{q}_{\Abs}(0) ~ + ~ O(\lambda). \nonumber
\end{eqnarray}

\noindent
For each $\phi_{j}(\lambda)$, we denote $\Phi_{j}(\lambda) := {\mathcal P}_{\Abs}^{q}(\lambda) \phi_{j}(\lambda)$.
Then,

\begin{eqnarray}     \label{E:2.14}
Q^{q}_{\Abs}(\lambda) (\phi_{j}(\lambda)) & = & - i^{\ast} \left( \iota_{\frac{\partial}{\partial_{u}}} d \Phi_{j}(\lambda) \right) ~ = ~ \theta_{j}(\lambda) \phi_{j}(\lambda),
\end{eqnarray}

\noindent
and hence $\{ \Phi_{1}(0), \cdots, \Phi_{\ell_{q}}(0) \}$ is also a basis for $\ker \Delta^{q}_{M, \Abs}$.
The Green Theorem (\ref{E:1.6}) shows that

\begin{eqnarray*}
0 & = & \langle (\Delta^{q}_{M} + \lambda) \Phi_{j}(\lambda),  \Phi_{k}(0) \rangle  =
\langle \Delta^{q}_{M} \Phi_{j}(\lambda),  \Phi_{k}(0) \rangle  +  \lambda  \langle \Phi_{j}(\lambda),  \Phi_{k}(0) \rangle   \\
& = & \langle d \Phi_{j}(\lambda),  d \Phi_{k}(0) \rangle  +  \langle \delta \Phi_{j}(\lambda),  \delta \Phi_{k}(0) \rangle  -  \int_{Y} i^{\ast} \left( \Phi_{k}(0) \wedge \star_{M} d \Phi_{j}(\lambda) - \delta \Phi_{j}(\lambda) \wedge \star_{M} \Phi_{k}(0) \right)    \nonumber \\
& & \hspace{1.0 cm}  + ~ \lambda ~ \langle \Phi_{j}(\lambda),  \Phi_{k}(0) \rangle
\nonumber \\
& = & - \langle \phi_{k}(0), ~ Q^{q}_{\Abs}(\lambda) \phi_{j}(\lambda) \rangle_{Y}  +  \lambda  \langle \Phi_{j}(\lambda),  \Phi_{k}(0) \rangle_{M}   \nonumber  \\
& = & -  \theta_{j}(\lambda) \langle \phi_{k}(0), \phi_{j}(\lambda) \rangle_{Y}  +  \lambda  \langle \Phi_{j}(\lambda),  \Phi_{k}(0) \rangle_{M},          \nonumber
\end{eqnarray*}

\noindent
which leads to

\begin{eqnarray}   \label{E:2.15}
\lim_{\lambda \rightarrow 0} \frac{\theta_{j}(\lambda)}{\lambda} \delta_{jk} & = &
\langle \Phi_{j}(0), ~ \Phi_{k}(0) \rangle_{M}.
\end{eqnarray}

\noindent
We define $\ell_{q} \times \ell_{q}$ matrices ${\mathcal R} = \left( r_{ij} \right)$ and
${\mathcal S} = \left( s_{ij} \right)$ by

\begin{eqnarray}   \label{E:2.16}
r_{ij} & = & \langle \Phi_{i}(0), ~ \psi_{j}(0) \rangle_{M}, \qquad  s_{ij} ~ = ~ \langle \psi_{i}(0)|_{Y}, ~ \psi_{j}(0)|_{Y} \rangle_{Y},
\end{eqnarray}

\noindent
where $\psi_{i}(0)|_{Y}$ is equal to $i^{\ast} \psi_{i}(0)$ since $\psi_{i}(0) \in \Omega^{q}_{\Abs}(M)$.
Since $\Phi_{i}(0) = \sum_{k} r_{ik} \psi_{k}(0)$, we have

\begin{eqnarray}   \label{E:2.17}
\langle \Phi_{i}(0), ~ \Phi_{j}(0) \rangle_{M} & = & \sum_{a, b =1}^{\ell_{q}} \langle r_{ia} \psi_{a}(0), ~ r_{jb} \psi_{b}(0) \rangle ~ = ~
\left( {\mathcal R} {\mathcal R}^{T} \right)_{ij} ~ = ~ \lim_{\lambda \rightarrow 0} \frac{\theta_{i}(\lambda)}{\lambda} \delta_{ij},
\end{eqnarray}

\noindent
which shows that ${\mathcal R} {\mathcal R}^{T}$ is a diagonal matrix.
The above equalities show that

\begin{eqnarray}    \label{E:2.18}
\lim_{\lambda \rightarrow 0} \frac{\theta_{1}(\lambda) \cdots \theta_{\ell_{q}}(\lambda)}{\lambda^{\ell_{q}}} & = &
 \Pi_{j=1}^{\ell_{q}} \left( {\mathcal R} {\mathcal R}^{T} \right)_{jj}
~ = ~ \ddet ({\mathcal R} {\mathcal R}^{T}) ~ = ~ \ddet ({\mathcal R}^{2}).
\end{eqnarray}

\noindent
Since $\phi_{i}(0) = \Phi_{i}(0)|_{Y} = \sum_{k=1}^{\ell_{q}} r_{ik} \psi_{k}(0)|_{Y}$, we have

\begin{eqnarray*}
\psi_{i}(0)|_{Y} & = & \sum_{k=1}^{\ell_{q}} \langle \psi_{i}(0)|_{Y}, ~ \phi_{k}(0) \rangle_{Y} ~ \phi_{k}(0) ~ = ~
\sum_{k, a, b=1}^{\ell_{q}}  \langle \psi_{i}(0)|_{Y}, ~ r_{ka} \psi_{a}(0)|_{Y} \rangle_{Y} ~ r_{kb} \psi_{b}(0)|_{Y}  \\
& = & \sum_{b=1}^{\ell_{q}} \left( {\mathcal S} {\mathcal R}^{T} {\mathcal R} \right)_{ib} ~ \psi_{b}(0)|_{Y},
\end{eqnarray*}

\noindent
which shows that ${\mathcal S} {\mathcal R}^{T} {\mathcal R} = \Id$.
Hence,

\begin{eqnarray}     \label{E:2.19}
\ddet {\mathcal R}^{2} = \frac{1}{\ddet {\mathcal S}}.
\end{eqnarray}

\vspace{0.2 cm}
\noindent
From (\ref{E:2.18}) and (\ref{E:2.19}), we obtain the following result.

\begin{theorem}  \label{Theorem:2.4}
Let $M$ be an oriented $m$-dimensional compact Riemannian manifold with boundary $Y$. Then, for $0 \leq q \leq m-1$,
\begin{eqnarray*}
\ln \Det^{\ast} \Delta^{q}_{M, \Abs} - \ln \Det \Delta^{q}_{M, \Dir} ~ = ~ a_{0} - \ln \ddet {\mathcal S} + \ln \Det^{\ast} Q^{q}_{\Abs}(0).
\end{eqnarray*}

\noindent
Here $ a_{0}$ is the constant term in the asymptotic expansion of $- \ln \Det Q^{q}_{\Abs}(\lambda)$ for $\lambda \rightarrow \infty$.
If $M$ has a product structure near $Y$ so that $\Delta^{q}_{M, \Abs}$ is $- \partial_{y_{m}}^{2} + \Delta_{Y}$ on a collar neighborhood of $Y$, then $a_{0} = 0$.
\end{theorem}

\vspace{0.2 cm}

\begin{corollary}  \label{Corollary:2.5}
When $q=0$, $\Delta^{0}_{M, \Abs}$ is the Laplacian acting on smooth functions with the Neumann boundary condition on $Y$.
Then $\ell_{0} = \Dim \Ker \Delta^{0}_{M, \Abs} = 1$ and ${\mathcal S} = \left( \frac{\ell(Y)}{V(M)} \right)$, where $V(M)$ and $\ell(Y)$ are volumes of $M$ and $Y$, respectively.
In this case, Theorem \ref{Theorem:2.4} is rewritten as

\begin{eqnarray*}
 \ln \Det^{\ast} \Delta^{0}_{M, \Abs} - \ln \Det \Delta^{0}_{M, \Dir} & = & a_{0} + \ln \frac{V(M)}{\ell(Y)} + \ln \Det^{\ast} Q^{0}_{\Abs}(0).
\end{eqnarray*}
\end{corollary}

We should mention that
the constant term corresponding to $a_{0}$ in the BFK-gluing formula of zeta-determinants is zero when $M$ is an even dimensional manifold since the density is an odd function with respect to $\xi$.
However, the density for $a_{0}$ in Theorem \ref{Theorem:2.4} need not be an odd function so that $a_{0}$ may not be zero even though $M$ is an even dimensional manifold.
In the next section we are going to compute $a_{0}$ precisely when the dimension of $M$ is $2$ and $3$.

In the remaining part of this section, we are going to discuss the values of zeta functions at zero by considering the metric rescaling
from $g$ to $c^{2} g$ for $c > 0$ on $M$. It is well known \cite{Be} that

\begin{eqnarray}   \label{E:2.20}
\Delta^{q}_{M}(c^{2} g) & = & c^{-2} \Delta^{q}_{M}(g), \qquad \Delta^{q}_{M}(c^{2} g) + \lambda ~ = ~ c^{-2} \left( \Delta^{q}_{M}(g) + c^{2} \lambda \right).
\end{eqnarray}

\noindent
Lemma \ref{Lemma:2.2} is rewritten as

\begin{eqnarray}   \label{E:2.21}
 \ln \Det \left( \Delta^{q}_{M, \Abs}(c^{2} g) + \lambda \right) - \ln \Det \left( \Delta^{q}_{M, \Dir}(c^{2} g) + \lambda \right)
& = & P^{q}_{c^{2} g}(\lambda) + \ln \Det Q^{q}_{\Abs, c^{2} g}(\lambda)     \\
& = & \sum_{j=0}^{[(m-1)/2]} a_{j}(c^{2} g) \lambda^{j} ~ + ~ \ln \Det Q^{q}_{\Abs, c^{2} g}(\lambda),  \nonumber
\end{eqnarray}

\noindent
where $P^{q}_{g}(\lambda) = \sum_{j=0}^{[\frac{m-1}{2}]} a_{j}(g) \lambda^{j}$ and $P^{q}_{c^{2} g}(\lambda) = \sum_{j=0}^{[\frac{m-1}{2}]} a_{j}(c^{2} g) \lambda^{j}$.
The Dirichlet-to-Neumann operator $Q^{q}_{\Abs, c^{2} g}(\lambda)$ is described as follows.
For $\varphi \in \Omega^{q}(Y)$, we choose $\phi \in \Omega^{q}(M)$ such that

\begin{eqnarray*}
\left( \Delta^{q}_{M}(c^{2} g) + \lambda  \right) \phi & = & c^{-2} \left( \Delta^{q}_{M}(g) + c^{2} \lambda \right) \phi ~ = ~ 0, \qquad
i^{\ast} \phi  = \varphi, \qquad i^{\ast} \left( \iota_{\frac{1}{c} \frac{\partial}{\partial u}} \phi \right) = 0.
\end{eqnarray*}

\noindent
Then, $Q^{q}_{\Abs, c^{2} g}(\lambda) f$ is defined by

\begin{eqnarray*}
Q^{q}_{\Abs, c^{2} g}(\lambda) \varphi & = & - i^{\ast} \left( \iota_{\frac{1}{c} \frac{\partial}{\partial u}} d \phi \right)  ~ = ~ - \frac{1}{c} ~ i^{\ast} \left( \iota_{\frac{\partial}{\partial u}} d \phi \right)
 ~ = ~\frac{1}{c} Q^q_{\Abs, g}(c^{2} \lambda) \varphi,
\end{eqnarray*}

\noindent
which shows that

\begin{eqnarray}  \label{E:2.22}
Q^q_{\Abs, c^{2} g}(\lambda) & = & \frac{1}{c} Q^q_{\Abs, g}(c^{2} \lambda) .
\end{eqnarray}

\noindent
From (\ref{E:2.20}) and (\ref{E:2.22}), it follows that

\begin{eqnarray}                \label{E:2.23}
& & \ln \Det \left( \Delta^{q}_{M, \Abs/\Dir} (c^{2} g) + \lambda \right) ~ = ~ - 2 \ln c \cdot \zeta_{( \Delta^{q}_{M, \Abs/\Dir}(g) + c^{2} \lambda)}(0)
~ + ~ \ln \Det \left( \Delta^{q}_{M, \Abs/\Dir}(g) + c^{2} \lambda \right) ,  \nonumber \\
& & \ln \Det Q^{q}_{\Abs, c^{2} g}(\lambda)  ~ = ~ - \ln c \cdot \zeta_{Q^{q}_{\Abs, g}(c^{2} \lambda)}(0) ~ + ~ \ln \Det Q^{q}_{\Abs, g}(c^{2} \lambda).
\end{eqnarray}

\noindent
We use the above equalities to rewrite (\ref{E:2.21}) as

\begin{eqnarray}   \label{E:2.24}
& & - 2 \ln c \cdot \left\{ \zeta_{( \Delta^{q}_{M, \Abs}(g) + c^{2} \lambda )}(0) - \zeta_{( \Delta^{q}_{M, \Dir}(g) + c^{2} \lambda )}(0) \right\}  \nonumber\\
& & \hspace{4.0cm}+
 \left\{ \ln \Det ( \Delta^{q}_{M, \Abs}(g) + c^{2} \lambda ) - \ln \Det ( \Delta^{q}_{M, \Dir}(g) + c^{2} \lambda )  \right\}    \nonumber \\
& & =  - 2 \ln c \cdot \left\{ \zeta_{\left( \Delta^{q}_{M, \Abs}(g) + c^{2} \lambda \right)}(0) - \zeta_{\left( \Delta^{q}_{M, \Dir}(g) + c^{2} \lambda \right)}(0) \right\}
~ + ~ P^{q}_{g}(c^{2} \lambda) ~ + ~ \ln \Det Q^{q}_{\Abs, g}(c^{2} \lambda)      \nonumber  \\
& & =  P^{q}_{c^{2} g}(\lambda) - \ln c \cdot \zeta_{Q^{q}_{\Abs, g}(c^{2} \lambda)}(0) ~ + ~ \ln \Det Q^{q}_{\Abs, g}(c^{2} \lambda),
\end{eqnarray}

\noindent
which leads to the following result.

\begin{lemma}   \label{Lemma:2.6}
\begin{eqnarray*}
& & -  \ln c \cdot \left\{ 2 \left( \zeta_{\left( \Delta^{q}_{M, \Abs}(g) + c^{2} \lambda \right)}(0) - \zeta_{\left( \Delta^{q}_{M, \Dir}(g) + c^{2} \lambda \right)}(0) \right) - \zeta_{Q^{q}_{\Abs, g}(c^{2} \lambda)}(0) \right\} ~ = ~
P^{q}_{c^{2} g}(\lambda) ~ - ~  P^{q}_{g}(c^{2} \lambda) \\
& & \hspace{9.0 cm}  = ~ \sum_{j=0}^{[(m-1)/2]} \left( a_{j}(c^{2} g) - a_{j}(g) c^{2j} \right) \lambda^{j}.
\end{eqnarray*}
\end{lemma}

\vspace{0.2 cm}

From (\ref{E:2.9}) we have the following.

\begin{eqnarray}          \label{E:2.25}
& & \sigma \left( \left( \mu - Q^{q}_{\Abs, c^{2} g}(\lambda) \right)^{-1} \right) ~ \sim ~ \sum_{j=0}^{\infty} \widetilde{r}_{-1-j} ( y, \xi, \lambda, \mu ; c^{2} g),  \\
& & \sigma \left( \left( \mu - Q^{q}_{\Abs, c^{2} g}(\lambda) \right)^{-1} \right)
~ = ~ \sigma \left( \left( \mu - \frac{1}{c} Q^{q}_{\Abs, g}(c^{2} \lambda) \right)^{-1} \right)
~ = ~ \sigma \left( c \left( c \mu - Q^{q}_{\Abs, g}(c^{2} \lambda) \right)^{-1} \right)   \nonumber \\
& \sim & c \sum_{j=0}^{\infty} \widetilde{r}_{-1-j} ( y, \xi, c^{2} \lambda, c \mu ; g)
~ = ~ c \sum_{j=0}^{\infty} \widetilde{r}_{-1-j} \left( y, c \frac{1}{c} \xi, c^{2} \lambda, c \mu ; g\right)
~ = ~  c \sum_{j=0}^{\infty} c^{-1-j} \widetilde{r}_{-1-j} \left( y, \frac{1}{c} \xi, \lambda, \mu ; g\right)  \nonumber \\
& = & \sum_{j=0}^{\infty} c^{-j} \widetilde{r}_{-1-j} \left( y, \frac{1}{c} \xi, \lambda, \mu ; g\right),    \nonumber
\end{eqnarray}

\noindent
which shows that

\begin{eqnarray}   \label{E:2.26}
\widetilde{r}_{-1-j} ( y, \xi, \lambda, \mu ; c^{2} g) & = &
c^{-j} \widetilde{r}_{-1-j} \left( y, \frac{1}{c} \xi, \lambda, \mu ; g\right).
\end{eqnarray}

\noindent
By (\ref{E:2.10}) the density $\pi_{j}(y ; c^{2} g)$ for $\pi_{j} (c^{2} g)$ is given by

\begin{eqnarray}   \label{E:2.27}
& & \pi_{j}(y ; c^{2} g) ~ = ~ - \frac{\partial}{\partial s}\bigg|_{s=0}  \left( \frac{1}{(2 \pi)^{m-1}} \int_{T_{y}^{\ast}{Y}}
\frac{1}{2 \pi i} \int_{\gamma} \mu^{-s} \Tr \widetilde{r}_{-1-j} \left( y, \xi, \frac{\lambda}{|\lambda|}, \mu ; c^{2} g \right) ~ d\mu ~ d \xi(c^{2} g) \right) , \nonumber \\
& &  \pi_{j}(c^{2} g) ~ = ~ \int_{Y} \pi_{j}(y ; c^{2} g) ~ d \vol (Y; c^{2} g).
\end{eqnarray}

\noindent
Using (\ref{E:2.27}) with the following relation

\begin{eqnarray}   \label{E:2.28}
d \xi(c^{2} g) & = & c^{-(m-1)} d \xi(g), \qquad  d \vol(Y; c^{2} g) ~ = ~ c^{m-1} d \vol(Y; g),
\end{eqnarray}

\noindent
we have the following result.

\begin{lemma}   \label{Lemma:2.7}
\begin{eqnarray*}
\pi_{k}(y ; c^{2} g) & = &  c^{-k} \pi_{k}(y ; g),  \qquad
\pi_{k}(c^{2} g) ~ = ~ c^{m-1-k} \pi_{k}(g).
\end{eqnarray*}
\end{lemma}

\vspace{0.2 cm}
\noindent
Lemma \ref{Lemma:2.3} shows that the coefficient $a_{k}(c^{2} g)$ of the polynomial $P^{q}_{c^{2} g}(\lambda)$ is given by

\begin{eqnarray}   \label{E:2.29}
a_{k}(c^{2} g) & = &  - \pi_{m-1-2k}(c^{2} g) - ({\frak a}_{m-2k}(c^{2} g) - {\frak b}_{m-2k}(c^{2} g)) ~ \frac{d}{ds} \left( \frac{\Gamma(s-k)}{\Gamma(s)} \right)\bigg|_{s=0}    \\
& = & c^{2k} a_{k}(g),    \nonumber
\end{eqnarray}

\noindent
where we used the fact that ${\frak a}_{k}(c^{2} g) = c^{m-k} {\frak a}_{k}(g)$ and ${\frak b}_{k}(c^{2} g) = c^{m-k} {\frak b}_{k}(g)$ (Theorem 3.1.9 in \cite{Gi3} or (4.2.5) in \cite{Ki}).
Hence, $P^{q}_{c^{2} g}(\lambda) ~ - ~  P^{q}_{g}(c^{2} \lambda) = 0$ in Lemma \ref{Lemma:2.6}.
Replacing $\lambda$ with $\frac{1}{c^{2}} \lambda$,
we obtain the following result.

\begin{theorem}  \label{Theorem:2.8}
For $\lambda > 0$, we obtain the following equality:
\begin{eqnarray*}
\zeta_{Q^{q}_{\Abs, g}(\lambda)}(0) & = & 2 \left\{ \zeta_{\left( \Delta^{q}_{M, \Abs}(g) + \lambda \right)}(0) - \zeta_{\left( \Delta^{q}_{M, \Dir}(g) + \lambda \right)}(0) \right\} .
\end{eqnarray*}
If $\Dim \Ker \Delta^{q}_{M, \Abs}(g) = \Dim \Ker Q^{q}_{\Abs, g}(0) = \ell_{q}$, we obtain the following equality by taking $\lambda \rightarrow 0$:
\begin{eqnarray*}
\zeta_{Q^{q}_{\Abs, g}(0)}(0) + \ell_{q}  & = & 2 \left\{ \left( \zeta_{\Delta^{q}_{M, \Abs}(g)}(0) + \ell_{q} \right) - \zeta_{\Delta^{q}_{M, \Dir}(g)}(0) \right\}.
\end{eqnarray*}
\end{theorem}

\vspace{0.2 cm}

The following heat trace asymptotic expansion is well known \cite{GS,Li,PS}.

\begin{eqnarray}    \label{E:2.30}
\Tr e^{-t Q^{q}_{\Abs}(0)} \sim \sum_{j=0}^{\infty} v_{j} t^{-(m-1)-j} + \sum_{j=1}^{\infty} (w_{j} \ln t + z_{j}) t^{j},
\end{eqnarray}

\noindent
where the $v_{j}$'s and $w_{j}$'s are locally computed and the $z_{j}$'s are not. The second statement of Theorem \ref{Theorem:2.8} can be rewritten as follows.

\begin{corollary}   \label{Corollary:2.9}
\begin{eqnarray*}
v_{m-1} & = & 2 \left( {\frak a}_{m} - {\frak b}_{m} \right).
\end{eqnarray*}
\end{corollary}

\vspace{0.3 cm}

Let $\kappa_{1}(y), \cdots, \kappa_{m-1}(y)$ be the principal curvatures of $Y$ in $M$ at $y \in Y$. We define the $r$-mean curvature $H_{r}$ by

\begin{eqnarray}   \label{E:2.31}
H_{r}(y) & = & \frac{1}{\binom{m-1}{r}} \sigma_{r}(\kappa_{1}, \cdots, \kappa_{m-1})
~ = ~ \frac{r ! (m-1-r)!}{(m-1)!} \sigma_{r}(\kappa_{1}, \cdots, \kappa_{m-1}),
\end{eqnarray}

\noindent
where $\sigma_{r} : {\mathbb R}^{m-1} \rightarrow {\mathbb R}$ is the $r$-th elementary symmetric polynomial defined by
$\sigma_{r}(u_{1}, \cdots. u_{m-1}) = \sum_{1 \leq i_{1} < \cdots < i_{r} \leq m-1} u_{i_{1}} \cdots u_{i_{r}}$ \cite{ALM}.
For example, for $m \geq 3$,
$H_{1}(y)$ and $H_{2}(y)$ are

\begin{eqnarray}    \label{E:2.32}
H_{1}(y) & = & \frac{1}{m-1} \sum_{\alpha = 1}^{m-1} \kappa_{\alpha}(y),
\qquad H_{2}(y) ~ = ~ \frac{2}{(m-1)(m-2)} \sum_{1 \leq \alpha < \beta \leq m-1} \kappa_{\alpha}(y) \kappa_{\beta}(y).
\end{eqnarray}

\noindent
When $m = 3$, the following equality was proved in Lemma 3.2 of \cite{KL4}.

\begin{eqnarray}    \label{E:2.33}
\sum_{\alpha=1}^{2} R^{M}_{\alpha 3 \alpha 3}(y) & = & - \Ric^{M}_{33} ~ = ~ - \frac{1}{2} \tau_{M}(y) + \frac{1}{2} \tau_{Y}(y) - H_{2}(y),
\end{eqnarray}

\noindent
where $R^{M}_{\alpha 3 \alpha 3}(y)$ and $\Ric^{M}_{33}$ are defined in (\ref{E:3.100}) below.
For $m = 2, ~~ 3$, ${\frak a}_{m} - {\frak b}_{m}$ can be computed concretely by using  Theorem 3.4.1 and Theorem 3.6.1 in \cite{Gi3} or Section 4.2 and 4.5 in \cite{Ki}, which together with (\ref{E:2.33}) and (\ref{E:3.78}) below yields the following result.

\begin{corollary}   \label{Corollary:2.10}
Let $(M, Y ; g)$ be an $m$-dimensional compact oriented Riemannian manifold with boundary $Y$.
We define $Q^{q}_{\Abs}(0)$ on $\Omega^{q}(Y)$ as above and denote by
$\tau_{M}$ and $\tau_{Y}$ the scalar curvatures of $M$ and $Y$, respectively.
If $m = 2$, then

\begin{eqnarray*}
\zeta_{Q^{q}_{\Abs}(0)}(0) + \ell_{q}
& = &  \begin{cases}  0 & \text{if} \quad q = 0  \\
- \frac{1}{\pi} \int_{Y} \kappa(y) ~ dy & \text{if} \quad q = 1 .  \end{cases}
\end{eqnarray*}

\noindent
If $m = 3$, then

\begin{eqnarray*}
\zeta_{Q^{q}_{\Abs}(0)}(0) + \ell_{q}
& = &  \begin{cases}
\frac{1}{4 \pi} \int_{Y} \Big\{ \frac{1}{8} \tau_{M} + \frac{1}{24} \tau_{Y} +  \frac{1}{4} H_{1}^{2} \Big\} dy & \text{if} \quad q = 0  \\
\frac{1}{4 \pi} \int_{Y} \Big\{ - \frac{1}{4} \tau_{M} - \frac{5}{12} \tau_{Y} +  \frac{1}{2} H_{1}^{2}  \Big\} dy  & \text{if} \quad q = 1 \\
\frac{1}{4 \pi} \int_{Y} \Big\{ - \frac{3}{8} \tau_{M} + \frac{13}{24} \tau_{Y} +  \frac{1}{4} H_{1}^{2}  \Big\} dy  & \text{if} \quad q = 2 . \end{cases}
\end{eqnarray*}
\noindent

\end{corollary}

\noindent
{\it Remark} : When $q = 0$, the above result is obtained in Theorem 1.5 of \cite{PS} or Theorem 5.1 of \cite{Li}.

\vspace{0.2 cm}
\noindent
{\it Example 2.11} : For a closed Riemannian manifold $N$, we consider a Riemannian product $M = [0, ~ a] \times N$. Let $\Delta^{q}_{M, \Abs}$ and $\Delta^{q}_{M, \Dir}$ be the Laplacian $- \frac{\partial^{2}}{\partial u^{2}} + \left(\begin{array}{clcr} \Delta^{q}_{N} \\ \Delta^{q-1}_{N} \end{array} \right)$
on $M$ acting on smooth $q$-forms with the absolute and Dirichlet boundary conditions on $Y := \{ 0, a \} \times N$, respectively. Let ${\mathbb N} = \{ 1, 2, 3, \cdots \}$ be the set of all positive integers and
 ${\mathbb N}_{0} = {\mathbb N} \cup \{ 0 \}$.
The spectra of $\Delta^{q}_{M, \Abs}$ and $\Delta^{q}_{M, \Dir}$ are given by

\begin{eqnarray*}
\Spec\left( \Delta^{q}_{M, \Abs} \right) & = & \left\{ \lambda_{n} + \left( \frac{k \pi}{a} \right)^{2},
~~ \mu_{s} + \left( \frac{l \pi}{a} \right)^{2} ~ \bigg| ~ \lambda_{n} \in \Spec(\Delta^{q}_{N}), ~~
\mu_{s} \in \Spec(\Delta^{q-1}_{N}), ~~ k \in {\mathbb N}_{0}, ~~ l \in {\mathbb N} \right\},  \\
\Spec\left( \Delta^{q}_{M, \Dir} \right) & = & \left\{ \lambda_{n} + \left( \frac{k \pi}{a} \right)^{2},
~~ \mu_{s} + \left( \frac{l \pi}{a} \right)^{2} ~ \bigg| ~ \lambda_{n} \in \Spec(\Delta^{q}_{N}), ~~
\mu_{s} \in \Spec(\Delta^{q-1}_{N}), ~~ k \in {\mathbb N}, ~~ l \in {\mathbb N} \right\},
\end{eqnarray*}

\noindent
which shows that $~ \zeta_{\Delta^{q}_{M, \Abs}}(s) - \zeta_{\Delta^{q}_{M, \Dir}}(s) = \zeta_{\Delta_{N}^{q}}(s) ~$ and hence

\begin{eqnarray*}
\ln \Det^{\ast} \Delta^{q}_{M, \Abs} ~ - ~ \ln \Det \Delta^{q}_{M, \Dir} & = & \ln \Det^{\ast} \Delta^{q}_{N}, \qquad
\zeta_{\Delta^{q}_{M, \Abs}}(0) ~ - ~ \zeta_{\Delta^{q}_{M, \Dir}}(0) ~ = ~ \zeta_{\Delta^{q}_{N}}(0).
\end{eqnarray*}

\noindent
Simple computation shows that the spectrum of $Q^{q}_{\Abs}(0) : \Omega^{q}(N \times \{ 0 \}) \oplus \Omega^{q}(N \times \{ a \}) \rightarrow
\Omega^{q}(N \times \{ 0 \}) \oplus \Omega^{q}(N \times \{ a \})$ is given by

\begin{eqnarray*}
& & \Spec\left( Q^{q}_{\Abs}(0) \right)   \\
& = & \left\{ 0 \right\} ~ \cup ~ \left\{ \frac{2}{a} \right\} ~ \cup ~
\left\{ \sqrt{\lambda_{n}} \left( 1 + \frac{2}{e^{a \sqrt{\lambda_{n}}} - 1} \right), ~ \sqrt{\lambda_{n}} \left( 1 - \frac{2}{e^{a \sqrt{\lambda_{n}}} + 1} \right)
\mid 0 < \lambda_{n} \in \Spec(\Delta^{q}_{N}) \right\},  \nonumber
\end{eqnarray*}

\noindent
where the multiplicities of $0$ and $\frac{2}{a}$ are $\ell_{q} := \Dim \Ker H^{q}(N) := \Dim \Ker H^{q}(M)$. Hence,

\begin{eqnarray*}
\ln \Det^{\ast} Q^{q}_{\Abs}(0) & = & \ell_{q} \ln \frac{2}{a} + \ln \Det^{\ast} \Delta^{q}_{N} +
\sum_{0 < \lambda_{n} \in \Spec(\Delta^{q}_{N})} \left\{ \ln \left( 1 + \frac{2}{e^{a \sqrt{\lambda_{n}}} - 1} \right) +
\ln \left( 1 - \frac{2}{e^{a \sqrt{\lambda_{n}}} + 1} \right) \right\}    \\
& = & \ell_{q} \ln \frac{2}{a} + \ln \Det^{\ast} \Delta^{q}_{N},  \\
\zeta_{Q^{q}_{\Abs}(0)}(0) & = & \ell_{q} + \zeta_{\Delta^{q}_{N}}(0) + \zeta_{\Delta^{q}_{N}}(0) ~ = ~ \ell_{q} + 2 \zeta_{\Delta^{q}_{N}}(0).
\end{eqnarray*}

\noindent
Let $\{ \psi_{1}, \cdots, \psi_{\ell_{q}} \}$ be an orthonormal basis of $\Ker \Delta^{q}_{N}$. Then
$\{ \frac{1}{\sqrt{a}} \psi_{1}, \cdots, \frac{1}{\sqrt{a}} \psi_{\ell_{q}} \}$ is an orthonormal basis of $\Ker \Delta^{q}_{M, \Abs}$. Hence,

\begin{eqnarray*}
\langle  \frac{1}{\sqrt{a}} \psi_{i}, ~ \frac{1}{\sqrt{a}} \psi_{j} \rangle_{Y} & = &
\frac{1}{a} \langle  \psi_{i}, ~ \psi_{j} \rangle_{\{ 0 \} \times N} ~ + ~ \frac{1}{a} \langle  \psi_{i}, ~ \psi_{j} \rangle_{\{ a \} \times N} ~ = ~ \frac{2}{a} \delta_{ij}.
\end{eqnarray*}

\noindent
Since $\ln \ddet {\mathcal S} = \ell_{q} \ln \frac{2}{a}$ and $a_{0} = 0$, this result agrees with Theorem \ref{Theorem:2.4} and Theorem \ref{Theorem:2.8}.

\vspace{0.3 cm}

\section{The homogeneous symbol of $Q_{\Abs}^{q}(\lambda)$}

In this section we are going to compute the homogeneous symbol of $Q_{\Abs}^{q}(\lambda)$ in the boundary normal coordinate system defined below.
For $y_{0} \in Y$ and a small open neighborhood $V$ of $y_{0}$ in $Y$, we choose a normal coordinate system on $V$ with $y = (y_{1}, \cdots, y_{m-1})$ and $y_{0} = (0, \cdots, 0)$.
For $y \in Y$, we denote by $\gamma_{y}(u)$ the unit speed geodesic such that $\gamma_{y}^{\prime}(0)$ is an inward normal vector to $Y$.
Then, $(y, u) = (y_{1}, \cdots, y_{m-1}, u)$ gives a local coordinate system.
We will write $u = y_{m}$ for notational convenience.
For $1 \leq \alpha, ~ \beta, ~\gamma \leq m-1$, the metric satisfies

\begin{eqnarray}  \label{E:3.1}
g_{\alpha\beta}(y_{0}) = \delta_{\alpha\beta}, \qquad g_{\alpha\beta; \gamma}(y_{0}) = 0, \qquad g_{\alpha m}(y) = 0, \qquad g_{mm}(y) = 1 ,
\end{eqnarray}

\noindent
where $g_{\alpha\beta;k} := \frac{\partial}{\partial y_{k}} g_{\alpha\beta}$,  $1 \leq k \leq m$.
Moreover, we may choose the coordinate system such that

\begin{eqnarray}   \label{E:3.2}
g^{\alpha \beta; m}(y_{0}) & = & - ~ g_{\alpha \beta; m}(y_{0}) ~ = ~ \begin{cases} 2 \kappa_{\alpha} & \quad \text{for} \quad \alpha = \beta \\
0 & \quad \text{for} \quad \alpha \neq \beta , \end{cases}
\end{eqnarray}

\noindent
where the $\kappa_{\alpha}$'s ($1 \leq \alpha \leq m-1$) are the principal curvatures of $Y$ in $M$.
For simplicity, we are going to write$\frac{\partial}{\partial y_{k}}$ by $\partial_{y_{k}}$ for $1 \leq k \leq m$.
We denote by $\nabla^{M}$ the Levi-Civita connection on $M$ associated to $g$ and
denote by $\omega$ the connection form for $\nabla^{M}$ with respect to $\{ \partial_{y_{1}}, \cdots, \partial_{y_{m}} \}$ and put $\omega_{k} = \omega(\partial_{y_{k}})$.
For some endomorphism $E_{q}$ acting on $~\wedge^{q} T^{\ast}M ~$,
$~\Delta^{q}_{M} + \lambda$ is expressed as follows \cite{KL4,PS}:

\begin{eqnarray}    \label{E:3.3}
& & \Delta^{q}_{M} + \lambda  ~ = ~ - \Tr \left( \left( \nabla^{M} \right)^{2} \right) - E_{q}  \\
& = & - \partial_{y_{m}}^{2} \Id ~ + ~ \left( A(y, y_{m}) - ~ 2 \omega_{m} \right)  \partial_{y_{m}} ~ + ~  D \big( y, y_{m},\frac{\partial}{\partial y}, \lambda \big)
 - \left( \partial_{y_{m}} \omega_{m} + \omega_{m} \omega_{m} - A(x, y_{m}) \omega_{m} \right) ,   \nonumber
\end{eqnarray}

\noindent
where $\Id$ is an $\binom{m}{q} \times \binom{m}{q}$ identity matrix and

\begin{eqnarray}   \label{E:3.4}
& & A(y, y_{m}) ~ = ~ \left\{ - \frac{1}{2} \sum_{\alpha, \beta = 1}^{m-1} g^{\alpha\beta}(y, y_{m}) ~ g_{\alpha\beta;m}(y, y_{m}) \right\} \Id,  \\
& & D\left(y, y_{m}, \frac{\partial}{\partial y}, \lambda \right) ~ = ~
\left\{ \left( - \sum_{\alpha, \beta = 1}^{m-1} g^{\alpha\beta}(y, y_{m}) \partial_{y_{\alpha}} \partial_{y_{\beta}} + \lambda \right)\right. \label{E:3.5}   \\
& &\left. \hspace{1.0 cm} -\sum_{\alpha, \beta = 1}^{m-1} \left( \frac{1}{2} g^{\alpha\beta}(y, y_{m}) (\partial_{y_{\alpha}} \ln |g|(y, y_{m}) ) +
g^{\alpha\beta; \alpha}(y, y_{m})) \right) \partial_{y_{\beta}} \right\} \Id   \nonumber  \\
& & \hspace{1.0 cm} - 2 \sum_{\alpha, \beta = 1}^{m-1} g^{\alpha \beta}(y, y_{m}) \omega_{\alpha} \partial_{y_{\beta}}
- \sum_{\alpha,\beta = 1}^{m-1} g^{\alpha\beta} (y, y_{m}) \left( \partial_{y_{\alpha}} \omega_{\beta} + \omega_{\alpha} \omega_{\beta} - \sum_{\gamma =1}^{m-1}\, \Gamma_{\alpha\beta}^{\gamma} \omega_{\gamma} \right) - E_{q} ,     \nonumber
\end{eqnarray}

\noindent
We use the Weitzenb\"ock formula (for example, Lemma 4.1.2 in \cite{Gi1}) to describe $E_{q}$ explicitly.
It is known (\cite{Gi3}) that $E_{0} = 0$.
Let $\{e_{1}, \cdots, e_{m} \}$ and $\{e^{1}, \cdots, e^{m} \}$ be local orthonormal bases of $TM|_{U}$ and $T^{\ast}M|_{U}$
for some open set $U$ in $M$, respectively.
We denote by $R^{M}_{ijkl}$ and $\Ric^{M}_{ij}$ the Riemann curvature tensor and Ricci tensor on $M$ defined by

\begin{eqnarray} \label{E:3.100}
R^{M}_{ijkl} = \big\langle \nabla^{M}_{e_{i}} \nabla^{M}_{e_{j}} e_{k} - \nabla^{M}_{e_{j}} \nabla^{M}_{e_{i}} e_{k} - \nabla^{M}_{[e_{i}, e_{j}]} e_{k}, e_{l} \big\rangle, \qquad
\Ric^{M}_{ij} = \sum_{k=1}^{m} R^{M}_{ikkj}.
\end{eqnarray}

\noindent
Then,

\begin{eqnarray}    \label{E:3.101}
E_{1} = \bigg( - \Ric^{M}_{ij} \bigg)_{1 \leq i, j \leq m}.
\end{eqnarray}

\noindent
For later use, we compute $E_{2}$ for $m = 3$ with respect to a local orthonormal basis
$\{ e^{1} \wedge e^{2}, e^{3} \wedge e^{1}, e^{3} \wedge e^{2} \}$, which is given by

\begin{eqnarray}    \label{E:3.102}
E_{2} = \left( \begin{array}{clcr} - \Ric^{M}_{33} & - R^{M}_{2113} & R^{M}_{1223} \\ - R^{M}_{2113} & - \Ric^{M}_{22} & R^{M}_{1332} \\
R^{M}_{1223} & R^{M}_{1332} & - \Ric^{M}_{11} \end{array} \right).
\end{eqnarray}

\vspace{0.2 cm}

Since $Y$ is compact, we can choose a uniform constant $\epsilon_{0} > 0$ such that $\gamma_{x}(u)$ is well defined for $0 \leq u \leq \epsilon_{0}$.
Then,

\begin{eqnarray}    \label{E:3.6}
U_{\epsilon_{0}} & := & \{ (y, y_{m}) \mid y \in Y, ~~0 \leq y_{m} < \epsilon_{0} \}
\end{eqnarray}

\noindent
is a collar neighborhood of $Y$.
We note that for a fixed $y_{m}$ in $[0, \epsilon_{0})$,

\begin{eqnarray}     \label{E:3.7}
Y_{y_{m}} & := & \{ (y, y_{m}) \mid y \in Y \}
\end{eqnarray}

\noindent
is a submanifold of $M$ diffeomorphic to $Y$, and it is the $y_{m}$-level of $Y$.
For $0 < y_{m} < \epsilon$, we denote

\begin{eqnarray}     \label{E:3.8}
M_{y_{m}} & := & M - \cup_{0 \leq u < y_{m}} Y_{u} ,
\end{eqnarray}

\noindent
and denote by $i_{y_{m}} : Y_{y_{m}} \rightarrow M_{y_{m}}$ the natural inclusion.
We also denote by $\Delta^{q}_{M_{y_{m}}}$ the Hodge-De Rham Laplacian $\Delta^{q}_{M}$ restricted to $M_{y_{m}}$. For each $0 \leq y_{m} < \epsilon_{0}$, we define $Q^{q}_{\Abs, y_{m}}(\lambda) : \Omega^{q}(Y_{y_{m}}) \rightarrow \Omega^{q}(Y_{y_{m}})$ and
$Q^{q-1}_{\Rel, y_{m}}(\lambda) : \Omega^{q-1}(Y_{y_{m}}) \rightarrow \Omega^{q-1}(Y_{y_{m}})$ in the same way as $Q^{q}_{\Abs}(\lambda)$ and $Q^{q-1}_{\Rel}(\lambda)$.
Indeed, for $\alpha_{y_{m}}(y) \in \Omega^{q}(Y_{y_{m}})$ and $\beta_{y_{m}}(y) \in \Omega^{q-1}(Y_{y_{m}})$, we choose
$\phi_{y_{m}} \in \Omega^{q}(M_{y_{m}})$, $\psi_{y_{m}} \in \Omega^{q}(M_{y_{m}})$
satisfying

\begin{eqnarray}    \label{E:3.9}
& & (\Delta^{q}_{M_{y_{m}}} + \lambda) \phi_{y_{m}} = 0, \qquad i_{y_{m}}^{\ast} \phi_{y_{m}} = \alpha_{y_{m}},
\qquad i^{\ast} \iota_{\partial_{y_{m}}} \phi_{y_{m}} = 0,  \\
& & (\Delta^{q}_{M_{y_{m}}} + \lambda) \psi_{y_{m}} = 0, \qquad i_{y_{m}}^{\ast} \psi_{y_{m}} = 0,
\qquad i^{\ast} \iota_{\partial_{y_{m}}} \psi_{y_{m}} = \beta_{y_{m}}.   \nonumber
\end{eqnarray}

\noindent
We define

\begin{eqnarray}     \label{E:3.10}
Q^{q}_{\Abs, y_{m}}(\lambda) (\varphi_{y_{m}}) ~ := ~ - i_{y_{m}}^{\ast} \iota_{\partial_{y_{m}}} d \phi_{y_{m}}, \qquad
Q^{q}_{\Rel, y_{m}}(\lambda) (\varphi_{y_{m}}) ~ := ~ i_{y_{m}}^{\ast} \left( \delta \psi_{y_{m}} \right).
\end{eqnarray}

\noindent
Using local coordinates on $U_{\epsilon_{0}}$, with multi-indices $i=(i_1,...,i_q)$, $j=(j_1,...,j_{q-1})$,
$k=(k_1,...,k_q)$, and $l=(l_1,...,l_{q-1})$,
we write $\phi_{y_{m}}(y, y_{m})$ and $\psi_{y_{m}}(y,y_{m})$ as

\begin{eqnarray}   \label{E:3.11}
\phi_{y_{m}}(y, y_{m}) ~ = ~ \sum_{i} \phi_{1, i}(y, y_{m}) ~ dy_{i_{1}} \wedge \cdots \wedge dy_{i_{q}} +
\sum_{j} \phi_{2, j}(y, y_{m}) ~ dy_{m} \wedge dy_{j_{1}} \wedge \cdots \wedge dy_{j_{q-1}},  \\
\psi_{y_{m}}(y, y_{m}) ~ = ~ \sum_{k} \psi_{1, k}(y, y_{m}) ~ dy_{k_{1}} \wedge \cdots \wedge dy_{k_{q}} +
\sum_{l} \psi_{2, l}(y, y_{m}) ~ dy_{m} \wedge dy_{l_{1}} \wedge \cdots \wedge dy_{l_{q-1}},  \nonumber
\end{eqnarray}

\noindent
where

\begin{eqnarray*}
& & \alpha_{y_{m}}(y) ~ = ~ \sum_{i} \phi_{1, i}(y, y_{m})\big|_{Y_{y_{m}}} ~ dy_{i_{1}} \wedge \cdots \wedge dy_{i_{q}}  , \\
& & \beta_{y_{m}}(y) ~ = ~ \sum_{l} \psi_{2, l}(y, y_{m})\big|_{Y_{y_{m}}} ~ dy_{l_{1}} \wedge \cdots \wedge dy_{l_{q-1}} ,   \\
& & \phi_{2, j}\big|_{Y_{y_{m}}} ~ = ~ \psi_{1, k}\big|_{Y_{y_{m}}} ~ = ~ 0.
\end{eqnarray*}

\noindent
In this local coordinate system, $Q^{q}_{\Abs, y_{m}}(\lambda) (\varphi_{y_{m}})$ and
$Q^{q}_{y_{m}, \Rel}(\lambda) (\varphi_{y_{m}})$ can be rewritten as follows (cf. Definition \ref{Definition:1.1}).

\begin{eqnarray}   \label{E:3.12}
Q^{q}_{\Abs, y_{m}}(\lambda) (\alpha_{y_{m}}(y)) & = & - \sum_{i} \big( \partial_{y_{m}} \phi_{1, i}(y, y_{m})\big)\big|_{Y_{y_{m}}}  ~ dy_{i_{1}} \wedge \cdots \wedge dy_{i_{q}}, \\
Q^{q-1}_{\Rel, y_{m}}(\lambda) (\beta_{y_{m}}(y)) & = & - \sum_{l} \big( \partial_{y_{m}} \psi_{2, l}(y, y_{m})\big)\big|_{Y_{y_{m}}}  ~ dy_{l_{1}} \wedge \cdots \wedge dy_{l_{q-1}}.    \nonumber
\end{eqnarray}

\noindent
When $y_{m} = 0$, $Q^{q}_{\Abs, 0}(\lambda)$ and $Q^{q-1}_{\Rel, 0}(\lambda)$ are equal to $Q^{q}_{\Abs}(\lambda)$ and $Q^{q-1}_{\Rel}(\lambda)$, respectively.

We next define auxiliary operators ${\mathcal T}^{q}_{\Abs, y_{m}}(\lambda) : \Omega^{q}(Y_{y_{m}}) \rightarrow \Omega^{q-1}(Y_{y_{m}})$ and
${\mathcal T}^{q-1}_{\Rel, y_{m}}(\lambda) : \Omega^{q-1}(Y_{y_{m}}) \rightarrow \Omega^{q}(Y_{y_{m}})$ by

\begin{eqnarray}  \label{E:3.13}
{\mathcal T}^{q}_{\Abs, y_{m}}(\lambda) (\alpha_{y_{m}}(y)) & = &
- \sum_{j} \big( \partial_{y_{m}} \phi_{2, j}(y, y_{m})\big)\big|_{Y_{y_{m}}}  ~ dy_{j_{1}} \wedge \cdots \wedge dy_{j_{q-1}}, \\
{\mathcal T}^{q-1}_{\Rel, y_{m}}(\lambda) (\beta_{y_{m}}(y)) & = &
- \sum_{k} \big( \partial_{y_{m}} \psi_{1, k}(y, y_{m})\big)\big|_{Y_{y_{m}}}  ~ dy_{k_{1}} \wedge \cdots \wedge dy_{k_{q}}.   \nonumber
\end{eqnarray}

\noindent
We finally define $~~{\mathcal R}^{q}_{y_{m}}(\lambda) : \Omega^{q}(Y_{y_{m}}) \oplus \Omega^{q-1}(Y_{y_{m}}) \rightarrow
\Omega^{q}(Y_{y_{m}}) \oplus \Omega^{q-1}(Y_{y_{m}})~~$ by

\begin{eqnarray}  \label{E:3.14}
{\mathcal R}^{q}_{y_{m}}(\lambda) & = & \left( \begin{array}{clcr}  Q^{q}_{\Abs, y_{m}}(\lambda) & {\mathcal T}^{q-1}_{\Rel, y_{m}}(\lambda) \\  {\mathcal T}^{q}_{\Abs, y_{m}}(\lambda) & Q^{q-1}_{\Rel, y_{m}}(\lambda)  \end{array} \right).
\end{eqnarray}

\noindent
Using local coordinates on $U_{\epsilon_{0}}$, we write

\begin{eqnarray}    \label{E:3.15}
& & {\mathcal R}^{q}_{y_{m}}(\lambda) (\alpha_{y_{m}}, ~ \beta_{y_{m}})   \\
& = & \left(  Q^{q}_{\Abs, y_{m}}(\lambda) \alpha_{y_{m}} + {\mathcal T}^{q-1}_{\Rel, y_{m}}(\lambda) \beta_{y_{m}}, ~
{\mathcal T}^{q}_{\Abs, y_{m}}(\lambda) \alpha_{y_{m}} + Q^{q-1}_{\Rel, y_{m}}(\lambda) \beta_{y_{m}}  \right)     \nonumber \\
& = & \bigg( - \sum_{i} \big( \partial_{y_{m}} \phi_{1, i}(y, y_{m})\big)\big|_{Y_{y_{m}}}  ~ dy_{i_{1}} \wedge \cdots \wedge dy_{i_{q}}
~ - ~ \sum_{k} \big( \partial_{y_{m}} \psi_{1, k}(y, y_{m})\big)\big|_{Y_{y_{m}}}  ~ dy_{k_{1}} \wedge \cdots \wedge dy_{k_{q}} ,    \nonumber  \\
& &  - ~ \sum_{j} \big( \partial_{y_{m}} \phi_{2, j}(y, y_{m}) \big)\big|_{Y_{y_{m}}}  ~ dy_{j_{1}} \wedge \cdots \wedge dy_{j_{q-1}}
 - \sum_{l} \big( \partial_{y_{m}}  \psi_{2, l}(y, y_{m})\big)\big|_{Y_{y_{m}}}  ~ dy_{l_{1}} \wedge \cdots \wedge dy_{l_{q-1}} \bigg),  \nonumber
\end{eqnarray}

\noindent
where $\phi_{y_{m}} + \psi_{y_{m}} \in \Omega^{q}(M_{y_{m}})$ satisfies

\begin{eqnarray}    \label{E:3.16}
\left( \Delta^{q}_{M} + \lambda \right) \left( \phi_{y_{m}} + \psi_{y_{m}} \right) = 0, \quad
i_{y_{m}}^{\ast} \left( \phi_{y_{m}} + \psi_{y_{m}} \right) = \alpha_{y_{m}}, \quad
i_{y_{m}}^{\ast} \left( \iota_{\partial_{y_{m}}} \left( \phi_{y_{m}} + \psi_{y_{m}} \right) \right) = \beta_{y_{m}}.
\end{eqnarray}

\noindent
Then, ${\mathcal R}^{q}_{y_{m}}(\lambda)$ is an elliptic pseudodifferential operator of order $1$.

We can identify $Y_{y_{m}}$ with $Y := Y_{0}$ by the geodesic $\gamma_{y}(u)$ and regard ${\mathcal R}^{q}_{y_{m}}(\lambda)$ to be a one parameter family of operators defined on $\Omega^{q}(Y) \oplus \Omega^{q-1}(Y)$. We are going to take the derivative of ${\mathcal R}^{q}_{y_{m}}(\lambda)$ with respect to $y_{m}$ to obtain a Riccati type equation for ${\mathcal R}^{q}_{y_{m}}(\lambda)$, from which we can compute the homogeneous symbol of ${\mathcal R}^{q}_{y_{m}}(\lambda)$. This idea goes back to I. M. Gelfand.
The symbol of $Q^{q}_{\Abs}(\lambda)$ is obtained from the symbol of ${\mathcal R}^{q}_{y_{m}}(\lambda)$.

\vspace{0.2 cm}

We start from $~~\phi_{y_{m}}(y, y_{m}) + \psi_{y_{m}}(y, y_{m}) \in \Omega^{q}(M_{y_{m}})$.
We note that

\begin{eqnarray}    \label{E:3.17}
\partial_{y_{m}} \bigg( \phi_{y_{m}}(y, y_{m}) + \psi_{y_{m}}(y, y_{m}) \bigg)\big|_{Y_{y_{m}}}
& = & - {\mathcal R}^{q}_{y_{m}}(\lambda) \left( \phi_{y_{m}}(y, y_{m}) + \psi_{y_{m}}(y, y_{m}) \right)\big|_{Y_{y_{m}}}.
\end{eqnarray}

\noindent
We take the derivative with respect to $y_{m}$ again to obtain

\begin{eqnarray}    \label{E:3.18}
& & \partial_{y_{m}}^{2} \bigg( \phi_{y_{m}}(y, y_{m}) + \psi_{x_{m}}(y, y_{m}) \bigg)\big|_{Y_{y_{m}}} \\
& = &
- \bigg( \partial_{y_{m}} {\mathcal R}^{q}_{y_{m}}(\lambda) \bigg) \left( \phi_{y_{m}}(y, y_{m}) + \psi_{y_{m}}(y, y_{m}) \right)\big|_{Y_{y_{m}}}
~ + ~ {\mathcal R}^{q}_{y_{m}}(\lambda)^{2} \left( \phi_{y_{m}}(y, y_{m}) + \psi_{y_{m}}(y, y_{m}) \right)\big|_{Y_{y_{m}}}, \nonumber
\end{eqnarray}

\noindent
which together with (\ref{E:3.3}) leads to the following equality.

\begin{eqnarray}   \label{E:3.19}
& & \left\{ - \partial_{y_{m}}{\mathcal R}^{q}_{y_{m}}(\lambda) +
{\mathcal R}^{q}_{y_{m}}(\lambda)^{2} \right\}
 \big( \phi_{y_{m}}(y, y_{m}) + \psi_{y_{m}}(y, y_{m}) \big)\big|_{Y_{y_{m}}}  \\
& = & \left\{\left( A(y, y_{m}) - ~ 2 \omega_{m} \right)  \partial_{y_{m}} ~ + ~  D \big( y, y_{m}, \partial_{y}, \lambda \big)
 - \left( \partial_{y_{m}} \omega_{m} + \omega_{m} \omega_{m} - A(y, y_{m}) \omega_{m} \right)\right\} \big( \phi_{y_{m}}(y, y_{m}) +    \nonumber \\
& & \hspace{3.0 cm} \psi_{y_{m}}(y, y_{m}) \big)\big|_{Y_{y_{m}}}.   \nonumber
\end{eqnarray}

\noindent
Using (\ref{E:3.17}) again, we obtain the following result.

\begin{lemma}   \label{Lemma:3.1}
\begin{eqnarray*}
& & {\mathcal R}^{q}_{y_{m}}(\lambda)^{2} \\
& = & D \big( y, y_{m},\partial_{y}, \lambda \big)  -
\left( A(y, y_{m}) - ~ 2 \omega_{m} \right) {\mathcal R}^{q}_{y_{m}}(\lambda)  +
\partial_{y_{m}}{\mathcal R}^{q}_{y_{m}}(\lambda)  -
 \left( \partial_{y_{m}} \omega_{m} + \omega_{m} \omega_{m} - A(y, y_{m}) \omega_{m} \right).
\end{eqnarray*}
\end{lemma}

We now compute the homogeneous symbol in this coordinate system using the above lemma.
We denote the homogeneous symbol of ${\mathcal R}^{q}_{y_{m}}(\lambda)$ and
$D \big( y, y_{m},\partial_{y}, \lambda \big)$ by

\begin{eqnarray}   \label{E:3.20}
& & \sigma \left( {\mathcal R}^{q}_{y_{m}}(\lambda) \right)(y, y_{m}, \xi, \lambda) ~ \sim ~
\alpha_{1}(y, y_{m}, \xi, \lambda) + \alpha_{0}(y, y_{m}, \xi, \lambda) +
\alpha_{-1}(y, y_{m}, \xi, \lambda) + \cdots,  \\
& & \sigma \left( D \big( y, y_{m},\partial_{y}, \lambda \big) \right) ~ = ~
p_{2}(y, y_{m}, \xi, \lambda) + p_{1}(y, y_{m}, \xi) + p_{0}(y, y_{m}, \xi),   \nonumber
\end{eqnarray}

\noindent
where for an $\binom{m}{q} \times \binom{m}{q}$ identity matrix $\Id$, (\ref{E:3.5}) shows that

\begin{eqnarray}   \label{E:3.21}
& & p_{2}(y, y_{m}, \xi, \lambda) ~ = ~ \left( \sum_{\alpha, \beta = 1}^{m-1} g^{\alpha\beta}(y,  y_{m}) \xi_{\alpha} \xi_{\beta} + \lambda \right) \Id
 ~ = ~ \left( | \xi |^{2} + \lambda \right) \Id,   \\
& & p_{1}(y,  y_{m}, \xi) ~ = ~ - i
\sum_{\alpha, \beta = 1}^{m-1} \left( \frac{1}{2} g^{\alpha\beta}(y,  y_{m}) \partial_{y_{\alpha}} \ln |g|(y,  y_{m}) +
g^{\alpha\beta;\alpha}(y,  y_{m}) \right) \xi_{\beta} \Id
- 2 i \sum_{\alpha, \beta = 1}^{m-1} g^{\alpha\beta} \omega_{\alpha} \xi_{\beta} ,   \nonumber \\
& & p_{0}(y,  y_{m}, \xi) ~ = ~
- \sum_{\alpha, \beta = 1}^{m-1} g^{\alpha\beta} \left( \partial_{y_{\alpha}} \omega_{\beta} + \omega_{\alpha} \omega_{\beta} - \sum_{\gamma = 1}^{m-1} \Gamma_{\alpha\beta}^{\gamma} \omega_{\gamma} \right) - E_{q} . \nonumber
\end{eqnarray}

\noindent
The symbol of $\partial_{y_{m}} {\mathcal R}^{q}_{y_{m}}(\lambda)$ is given by

\begin{eqnarray}   \label{E:3.22}
\sigma \left( \partial_{y_{m}} {\mathcal R}^{q}_{y_{m}}(\lambda) \right) (y,y_m,\xi,\lambda)  \sim    \partial_{y_{m}} \alpha_{1}(y, y_{m}, \xi, \lambda)  +
\partial_{y_{m}} \alpha_{0}(y, y_{m}, \xi, \lambda) + \partial_{y_{m}} \alpha_{-1}(y,  y_{m}, \xi, \lambda) +  \cdots .
\end{eqnarray}

\noindent
It is well known \cite{Gi1,Sh} that for $D_{y} = \frac{1}{i} \partial_{y}$,

\begin{eqnarray}   \label{E:3.23}
\sigma \left( {\mathcal R}_{y_{m}}(\lambda)^{2} \right) & \sim & \sum_{k=0}^{\infty} \sum_{\stackrel{|\omega|+i+j=k}{i, j \geq 0}} \frac{1}{\omega !} \partial^{\omega}_{\xi} \alpha_{1-i}(y, y_{m}, \xi, \lambda) \cdot D_{y}^{\omega}\alpha_{1-j}(y, y_{m}, \xi, \lambda)       \\
& = & \alpha_{1}^{2} + \left(\partial_{\xi} \alpha_{1} \cdot D_{y} \alpha_{1} +
2 \alpha_{1} \cdot \alpha_{0} \right)     \nonumber \\
& & + ~ \left( 2 \alpha_{1} \alpha_{-1} + \alpha_{0}^{2} - i (\partial_{\xi} \alpha_{0} ) (\partial_{y} \alpha_{1} )- i (\partial_{\xi} \alpha_{1} )(\partial_{y} \alpha_{0}) - \sum_{|\omega|=2}
\frac{1}{\omega !} (\partial^{\omega}_{\xi} \alpha_{1})( \partial^{\omega}_{y} \alpha_{1}) \right) ~ + ~ \cdots . \nonumber
\end{eqnarray}

\noindent
Using Lemma \ref{Lemma:3.1} with (\ref{E:3.21}) - (\ref{E:3.23}), we can compute the homogeneous symbol of
${\mathcal R}^{q}_{y_{m}}(\lambda)$. For example, the first three terms are given as follows.

\begin{eqnarray}  \label{E:3.24}
& & \alpha_{1}(y, y_{m}, \xi, \lambda) ~ = ~ \sqrt{| \xi |^{2} + \lambda} ~ \Id,  \\
& & \alpha_{0}(y, y_{m}, \xi, \lambda) ~ = ~ \frac{1}{2 \sqrt{| \xi |^{2} + \lambda}} \left\{- \partial_{\xi} \alpha_{1} \cdot D_{y} \alpha_{1} + p_{1}
- \left( A(y, y_{m}) - ~ 2 \omega_{m} \right) \alpha_{1} + \partial_{y_{m}} \alpha_{1} \right\},    \nonumber \\
& & \alpha_{-1}(y, y_{m}, \xi, \lambda) ~ = ~ \frac{1}{2 \sqrt{|\xi|^2 + \lambda}}
\bigg\{ \sum_{|\omega|=2} \frac{1}{\omega !} (\partial_{\xi}^{\omega} \alpha_{1} )(\partial_{y}^{\omega} \alpha_{1} )+
i (\partial_{\xi} \alpha_{0})( \partial_{y} \alpha_{1} )+ i (\partial_{\xi} \alpha_{1})( \partial_{y} \alpha_{0}) - \alpha_{0}^{2}   \nonumber  \\
& &~~~~~~~~\hspace{3.0 cm}  ~ + ~ p_{0} - (A(y, y_{m}) - 2 \omega_{m}) \alpha_{0} + \partial_{y_{m}} \alpha_{0} - (\partial_{y_{m}} \omega_{m} + \omega_{m} \omega_{m} - A(y, y_{m}) \omega_{m}) \bigg\}.   \nonumber
\end{eqnarray}

Let $~ {\mathcal F}_{y_{m}} : \Omega^{q}(Y_{y_{m}}) \rightarrow \Omega^{q}(Y_{y_{m}}) \oplus \Omega^{q-1}(Y_{y_{m}}) ~$ and
$~ {\mathcal G}_{y_{m}} : \Omega^{q}(Y_{y_{m}}) \oplus \Omega^{q-1}(Y_{y_{m}}) \rightarrow \Omega^{q}(Y_{y_{m}}) ~$ be the natural inclusion and projection, respectively, {\it i.e.} ${\mathcal F}_{y_{m}} (\phi) = (\phi, 0)$ and ${\mathcal G}_{y_{m}} (\phi, \psi) = \phi$.
Then, by (\ref{E:3.14}) it follows that

\begin{eqnarray}    \label{E:3.25}
Q^{q}_{\Abs, y_{m}}(\lambda) & = &  {\mathcal G}_{y_{m}} \cdot {\mathcal R}^{q}_{y_{m}}(\lambda) \cdot {\mathcal F}_{y_{m}},
\end{eqnarray}

\noindent
which shows that the symbol of $Q^{q}_{\Abs, y_{m}}(\lambda)$ is given by

\begin{eqnarray}    \label{E:3.26}
\sigma \left( Q^{q}_{\Abs, y_{m}}(\lambda) \right) & = & ( \I,~ \Oo) ~ \big\{ \sigma \left( {\mathcal R}^{q}_{y_{m}}(\lambda) \right) \big\} ~ ( \I,~ \Oo)^{T},
\end{eqnarray}

\noindent
where $\I$ is the $\binom{m-1}{q} \times \binom{m-1}{q}$ identity matrix and $\Oo$ is the $\binom{m-1}{q} \times \binom{m-1}{q-1}$ zero matrix.

We consider the boundary normal coordinate system $(y, y_{m}) = (y_{1}, \cdots, y_{m-1}, y_{m})$ on a collar neighborhood of $Y$
introduced at the beginning of this section. For $y_{0} \in Y$, we denote $e_{i} := \partial_{x_{i}}(y_{0})$ and $e^{i} := dx_{i}(y_{0})$ for
$1 \leq i \leq m$. Eq.(\ref{E:3.1}) and (\ref{E:3.2}) show that $\{ e_{1}, \cdots, e_{m} \}$ and $\{ e^{1}, \cdots, e^{m} \}$
at $y_{0} \in Y$ satisfy the following relations.

\begin{eqnarray}    \label{E:3.55}
\nabla^{M}_{e_{\alpha}} e^{\beta} = \omega_{\alpha}(e^{\beta}) =  \kappa_{\alpha} \delta_{\alpha \beta} e^{m}, \qquad
\nabla^{M}_{e_{\alpha}} e^{m} = - \kappa_{\alpha} e^{\alpha},  \qquad
\nabla^{M}_{e_{m}} e^{\alpha} = \kappa_{\alpha} e^{\alpha}, \qquad \nabla^{M}_{e_{m}} e^{m} = 0,
\end{eqnarray}

\noindent
The following result is straightforward (cf. Lemma 1.5.4 of \cite{Gi3}).

\begin{lemma}    \label{Lemma:3.2}
For $1 \leq \alpha \leq m-1$ and $1 \leq i_{1} < \cdots < i_{q} \leq m-1$, the following equalities hold.

\begin{eqnarray*}
& & \omega_{\alpha} \left( e^{i_{1}} \wedge \cdots \wedge e^{i_{q}} \right) ~ = ~ \kappa_{\alpha} ~ e^{m} \wedge \left( \iota_{e_{\alpha}} e^{i_{1}} \wedge \cdots \wedge e^{i_{q}} \right), \\
& & \omega_{\alpha} \left( e^{m} \wedge e^{j_{1}} \wedge \cdots \wedge e^{j_{q-1}} \right) ~ = ~
- \kappa_{\alpha} ~ e^{\alpha} \wedge e^{j_{1}} \wedge \cdots \wedge e^{j_{q-1}},  \\
& & \omega_{m} \left( e^{i_{1}} \wedge \cdots \wedge e^{i_{q}} \right)
~ = ~ (\kappa_{i_{1}} + \cdots + \kappa_{i_{q}}) ~ e^{i_{1}} \wedge \cdots \wedge e^{i_{q}},  \\
& & \omega_{m} \left( e^{m} \wedge e^{j_{1}} \wedge \cdots \wedge e^{j_{q-1}} \right)
~ = ~ (\kappa_{j_{1}} + \cdots + \kappa_{j_{q-1}}) ~ e^{m} \wedge e^{j_{1}} \wedge \cdots \wedge e^{j_{q-1}}.
\end{eqnarray*}
\end{lemma}

\noindent
When $\Dim M = 3$, Lemma \ref{Lemma:3.2} is reduced to the following result.

\begin{corollary}   \label{Corollary:3.4}
Let $\Dim M = 3$.
For $p = 1$ and an ordered basis $\{ e^{1}, e^{2}, e^{3} \}$ of $T^{\ast}M\big|_{U}$, we can write
$\omega_{1}$, $\omega_{2}$ and $\omega_{m}$ ($m = 3$) by

\begin{eqnarray}   \label{E:3.27}
\omega_{1} & = & \left( \begin{array}{clcr} 0 & 0 & - \kappa_{1} \\ 0 & 0 & 0 \\ \kappa_{1} & 0 & 0 \end{array} \right), \qquad \omega_{1} ~ \omega_{1} ~ = ~ - \kappa_{1}^{2} \left( \begin{array}{clcr} 1 & 0 & 0 \\ 0 & 0 & 0 \\ 0& 0 & 1 \end{array} \right),  \\
\omega_{2} & = & \left( \begin{array}{clcr} 0 & 0 & 0 \\ 0 & 0 & - \kappa_{2} \\ 0 & \kappa_{2} & 0 \end{array} \right), \qquad \omega_{2} ~ \omega_{2} ~ = ~ - \kappa_{2}^{2} \left( \begin{array}{clcr} 0 & 0 & 0 \\ 0 & 1 & 0 \\ 0& 0 & 1 \end{array} \right),  \nonumber  \\
\omega_{m} & = & \left( \begin{array}{clcr} \kappa_{1} & 0 & 0 \\ 0 & \kappa_{2} & 0 \\ 0 & 0 & 0 \end{array} \right), \qquad \omega_{m} ~ \omega_{m} ~ = ~ \left( \begin{array}{clcr} \kappa_{1}^{2} & 0 & 0 \\ 0 & \kappa_{2}^{2} & 0 \\ 0& 0 & 0 \end{array} \right).  \nonumber
\end{eqnarray}

\noindent
For $p = 2$ and an ordered basis $\{ e^{1} \wedge e^{2}, ~ e^{3} \wedge e^{1}, ~ e^{3} \wedge e^{2} \}$ of $\wedge^{2} T^{\ast}M\big|_{U}$, we can write
$\omega_{1}$, $\omega_{2}$ and $\omega_{m}$ ($m = 3$) by

\begin{eqnarray}   \label{E:3.28}
\omega_{1} & = & \left( \begin{array}{clcr} 0 & 0 & - \kappa_{1} \\ 0 & 0 & 0 \\ \kappa_{1} & 0 & 0 \end{array} \right), \qquad \omega_{1} ~ \omega_{1} ~ = ~ - \kappa_{1}^{2} \left( \begin{array}{clcr} 1 & 0 & 0 \\ 0 & 0 & 0 \\ 0& 0 & 1 \end{array} \right),  \\
\omega_{2} & = & \left( \begin{array}{clcr} 0 & \kappa_{2} & 0 \\ - \kappa_{2} & 0 & 0 \\ 0 & 0 & 0 \end{array} \right), \qquad \omega_{2} ~ \omega_{2} ~ = ~ - \kappa_{2}^{2} \left( \begin{array}{clcr} 1 & 0 & 0 \\ 0 & 1 & 0 \\ 0& 0 & 0 \end{array} \right),  \nonumber  \\
\omega_{m} & = & \left( \begin{array}{clcr} \kappa_{1} + \kappa_{2} & 0 & 0 \\ 0 & \kappa_{1} & 0 \\ 0 & 0 & \kappa_{2} \end{array} \right), \qquad \omega_{m} ~ \omega_{m} ~ = ~ \left( \begin{array}{clcr} (\kappa_{1} + \kappa_{2})^{2} & 0 & 0 \\ 0 & \kappa_{1}^{2} & 0 \\
0& 0 & \kappa_{2}^{2} \end{array} \right).  \nonumber
\end{eqnarray}
\end{corollary}

\vspace{0.2 cm}
\noindent
We denote

\begin{eqnarray}    \label{E:3.29}
& & ( \I,~ \Oo) ~ E_{q} ~ ( \I,~ \Oo)^{T} ~ = ~ {\widetilde E}_{q}, \qquad
( \I,~ \Oo) ~ \omega_{m}(y, y_{m}) ~ ( \I,~ \Oo)^{T} ~ = ~ \widetilde{\omega}_{m}(y, y_{m}), \\
& & ( \I,~ \Oo) ~ \omega_{\alpha}(y, y_{m})  ~ ( \I,~ \Oo)^{T}
~ = ~ \widetilde{\omega}_{\alpha}(y, y_{m}), \qquad
 ( \I,~ \Oo) ~ \omega_{\alpha}(y, y_{m}) \omega_{\beta}(y, y_{m})  ~ ( \I,~ \Oo)^{T}
~ = ~ \widetilde{\omega_{\alpha} \omega_{\beta}}(y, y_{m}).    \nonumber
\end{eqnarray}

\noindent
When $m=3$, ${\widetilde E}_{1}$ and ${\widetilde E}_{2}$ are given by (\ref{E:3.101}) and (\ref{E:3.102}) as follows.

\begin{eqnarray}      \label{E:3.77}
{\widetilde E}_{1} & = &
\left( \begin{array}{clcr} - \Ric^{M}_{11} & - \Ric^{M}_{12} \\  - \Ric^{M}_{12} & - \Ric^{M}_{22}  \end{array} \right), \qquad
{\widetilde E}_{2} ~ = ~ \bigg( - \Ric^{M}_{33} \bigg).
\end{eqnarray}

\noindent
Moreover, we use (\ref{E:2.33}) to obtain the following equalities.

\begin{eqnarray}   \label{E:3.78}
\Tr {\widetilde E}_{1} & = & - \tau_{M} + \Ric^{M}_{33} ~ = ~ - \frac{1}{2} \big( \tau_{M} + \tau_{Y} \big) + H_{2},  \\
\Tr {\widetilde E}_{2} & = & - \Ric^{M}_{33} ~ = ~ - \frac{1}{2} \big( \tau_{M} - \tau_{Y} \big) - H_{2}.   \nonumber
\end{eqnarray}

\noindent
We also denote

\begin{eqnarray}    \label{E:3.30}
& & \widetilde{p}_{2}(y, y_{m}, \xi, \lambda) ~ = ~ \left( \sum_{\alpha, \beta = 1}^{m-1} g^{\alpha\beta}(y,  y_{m}) \xi_{\alpha} \xi_{\beta} + \lambda \right) \widetilde{\Id}
 ~ = ~ \left( | \xi |^{2} + \lambda \right) \widetilde{\Id},   \\
& & \widetilde{p}_{1}(y,  y_{m}, \xi) ~ = ~ - i
\sum_{\alpha, \beta = 1}^{m-1} \left( \frac{1}{2} g^{\alpha\beta}(y,  y_{m}) \partial_{y_{\alpha}} \ln |g|(y,  y_{m}) + g^{\alpha\beta;\alpha}(y,  y_{m}) \right) \xi_{\beta} \widetilde{\Id}
- 2 i \sum_{\alpha, \beta = 1}^{m-1} g^{\alpha\beta} \widetilde{\omega}_{\alpha} \xi_{\beta},   \nonumber \\
& & \widetilde{p}_{0}(y,  y_{m}, \xi) ~ = ~
- \sum_{\alpha, \beta = 1}^{m-1} g^{\alpha\beta} \left( \partial_{y_{\alpha}} \widetilde{\omega}_{\beta} +
\widetilde{\omega_{\alpha} \omega_{\beta}} -
\sum_{\gamma = 1}^{m-1} \Gamma_{\alpha\beta}^{\gamma} \widetilde{\omega}_{\gamma} \right) - \widetilde{E}_{q} , \nonumber  \\
& & \widetilde{A}(y, y_{m}) ~ = ~ \left\{ - \frac{1}{2} \sum_{\alpha, \beta = 1}^{m-1} g^{\alpha\beta}(y, y_{m}) ~ g_{\alpha\beta;m}(y, y_{m}) \right\} \widetilde{\Id},    \nonumber
\end{eqnarray}

\noindent
where $~\widetilde{\Id}~$ is the $\binom{m-1}{q} \times \binom{m-1}{q}$ identity matrix.

\vspace{0.2 cm}
\noindent
{\it Remark} : At $(y_{0}, 0) \in Y$, Lemma \ref{Lemma:3.2} (or Corollary \ref{Corollary:3.4}) shows that $\widetilde{\omega}_{\alpha}(y_{0}, 0)  =  0$, and hence $\widetilde{p}_{1}(y_{0}, 0) = 0$ by (\ref{E:3.1}).
Since $\widetilde{\omega_{\alpha} \omega_{\alpha}} \neq 0$ as shown in (\ref{E:3.27}) and (\ref{E:3.28}), it follows that

\begin{eqnarray}   \label{E:3.50}
\bigg(\widetilde{p}_{1}(y_{0}, 0)\bigg)^{2} & = & 0, \qquad
(\I, ~ \Oo) ~ p_{1}^{2}(y_{0}, 0, \xi) ~ (\I, ~ \Oo)^{T} ~ = ~ - 4 ~ \sum_{\alpha, \beta=1}^{m-1}
\widetilde{\omega_{\alpha} \omega_{\beta}} (y_{0}, 0) ~ \xi_{\alpha} \xi_{\beta} .
\end{eqnarray}

\vspace{0.2 cm}

\noindent
Using Lemma \ref{Lemma:3.1} with (\ref{E:3.26}), we can compute the homogeneous symbol of
$Q^{q}_{\Abs, y_{m}}(\lambda)$, whose first three terms are given as follows (cf. (1.7)-(1.9) in \cite{LU}, (2.2)-(2.3) in \cite{PS} for $q=0$).

\begin{theorem}  \label{Theorem:3.4}
In the boundary normal coordinate system given at the beginning of this section, we denote the homogeneous symbol of $Q^{q}_{\Abs, y_{m}}(\lambda)$ by
\begin{eqnarray*}
\sigma \left( Q^{q}_{\Abs, y_{m}}(\lambda) \right)(y, y_{m}, \xi, \lambda) ~ \sim ~
\widetilde{\alpha}_{1}(y, y_{m}, \xi, \lambda) + \widetilde{\alpha}_{0}(y, y_{m}, \xi, \lambda) +
\widetilde{\alpha}_{-1}(y, y_{m}, \xi, \lambda) + \cdots.
\end{eqnarray*}
Then,
\begin{eqnarray*}
& & \widetilde{\alpha}_{1}(y, y_{m}, \xi, \lambda)
~ = ~ ( \I,~ \Oo) ~ \alpha_{1} ~ ( \I,~ \Oo)^{T}
 ~ = ~ \sqrt{| \xi |^{2} + \lambda} ~ \widetilde{\Id},  \\
& & \widetilde{\alpha}_{0}(y, y_{m}, \xi, \lambda) ~ = ~ ( \I,~ \Oo) ~ \alpha_{0} ~ ( \I,~ \Oo)^{T}  \\
& & \hspace{2.5 cm} = ~~ \frac{1}{2 \sqrt{| \xi |^{2} + \lambda}} \left\{- \partial_{\xi}  \widetilde{\alpha}_{1} \cdot D_{y} \widetilde{\alpha}_{1} + \widetilde{p}_{1}
- \left( \widetilde{A}(y, y_{m}) - ~ 2 \widetilde{\omega}_{m} \right) \widetilde{\alpha}_{1} + \partial_{y_{m}} \widetilde{\alpha}_{1} \right\},     \\
& & \widetilde{\alpha}_{-1}(y, y_{m}, \xi, \lambda) ~ = ~ ( \I,~ \Oo) ~ \alpha_{-1} ~ ( \I,~ \Oo)^{T}  \\
& & \hspace{0.5 cm} = ~ \frac{1}{2 \sqrt{|\xi|^2 + \lambda}}
\bigg\{ \sum_{|\omega|=2} \frac{1}{\omega !} (\partial_{\xi}^{\omega} \widetilde{\alpha}_{1} )(\partial_{y}^{\omega} \widetilde{\alpha}_{1} )+
i (\partial_{\xi} \widetilde{\alpha}_{0})( \partial_{y} \widetilde{\alpha}_{1} )+ i (\partial_{\xi} \widetilde{\alpha}_{1})( \partial_{y} \widetilde{\alpha}_{0})
 -   ( \I,~ \Oo) ~ \alpha_{0}^{2} ~ ( \I,~ \Oo)^{T}   \\
& & \hspace{2.5 cm} + ~ \widetilde{p}_{0}  - \big(\widetilde{A}(y, y_{m}) - 2 \widetilde{\omega}_{m} \big) \widetilde{\alpha}_{0} + \partial_{y_{m}} \widetilde{\alpha}_{0} - \big(\partial_{y_{m}} \widetilde{\omega_{m}} + \widetilde{\omega_{m}} \widetilde{\omega_{m}} - \widetilde{A}(y, y_{m}) \widetilde{\omega}_{m} \big) \bigg\},
\end{eqnarray*}
where at $(x, 0) \in Y$,

\begin{eqnarray*}
( \I,~ \Oo) ~ \alpha_{0}^{2}(y, 0) ~ ( \I,~ \Oo)^{T} & = & \bigg( \widetilde{\alpha}_{0}(y, 0) \bigg)^{2} ~ - ~
\frac{1}{| \xi |^{2} + \lambda} \sum_{\alpha=1}^{m-1} ~ \widetilde{\omega_{\alpha} \omega_{\beta}}(y, 0) ~ \xi_{\alpha} \xi_{\beta}.
\end{eqnarray*}
\end{theorem}

We next denote the homogeneous symbol of the resolvent $\left( \mu - Q^{q}_{\Abs}(\lambda) \right)^{-1}$ by

\begin{eqnarray}    \label{E:3.31}
\sigma \left( (\mu - Q^{q}_{\Abs}(\lambda))^{-1} \right)(y, \xi, \lambda, \mu) & \sim & \widetilde{r}_{-1}(y, \xi, \lambda, \mu) + \widetilde{r}_{-2}(y, \xi, \lambda, \mu) +
\widetilde{r}_{-3}(y, \xi, \lambda, \mu) + \cdots.
\end{eqnarray}

\noindent
Then,

\begin{eqnarray}    \label{E:3.32}
\widetilde{r}_{-1}(y, \xi, \lambda, \mu) & = & \left( \mu - \sqrt{|\xi|^2 + \lambda} \right)^{-1} \widetilde{\Id},  \\
\widetilde{r}_{-1-j}(y, \xi, \lambda, \mu) & = & \left( \mu - \sqrt{|\xi|^2 + \lambda} \right)^{-1} \sum_{k=0}^{j-1} \sum_{|\omega|+l+k=j} \frac{1}{\omega!}
\partial_{\xi}^{\omega} \widetilde{\alpha}_{1-l} D_{y}^{\omega} \widetilde{r}_{-1-k},    \nonumber
\end{eqnarray}

\noindent
which shows that the first three terms are given as follows.

\begin{eqnarray}    \label{E:3.33}
\widetilde{r}_{-1} & = & \left( \mu - \sqrt{|\xi|^2 + \lambda} \right)^{-1} \widetilde{\Id},  \\
\widetilde{r}_{-2} & = & \left( \mu - \sqrt{|\xi|^2 + \lambda} \right)^{-1}
\left\{ \partial_{\xi} \widetilde{\alpha}_{1} \cdot D_{y} \widetilde{r}_{-1} + \widetilde{\alpha}_{0} \cdot \widetilde{r}_{-1} \right\} , \nonumber  \\
\widetilde{r}_{-3} & = & \left( \mu - \sqrt{|\xi|^2 + \lambda} \right)^{-1} \Big\{ \sum_{|\omega|=2} \frac{1}{\omega !} \partial_{\xi}^{\omega} \widetilde{\alpha}_{1} \cdot D_{y}^{\omega} \widetilde{r}_{-1}
+ \partial_{\xi} \widetilde{\alpha}_{1} \cdot D_{y} \widetilde{r}_{-2} + \partial_{\xi} \widetilde{\alpha}_{0} \cdot D_{y} \widetilde{r}_{-1}  + \widetilde{\alpha}_{0} \cdot \widetilde{r}_{-2}
+ \widetilde{\alpha}_{-1} \cdot \widetilde{r}_{-1} \Big\}.  \nonumber
\end{eqnarray}

\vspace{0.3 cm}

\section{The constant term $a_{0}$ in $2$ and $3$ dimensional manifolds}

\vspace{0.3 cm}

In this section we are going to compute $a_{0}$ in Theorem \ref{Theorem:2.4} in terms of curvature tensors on $Y$ when $\Dim Y = 1$ and $2$. By Lemma \ref{Lemma:2.3} with (\ref{E:2.10}), $a_{0}(y)$ is expressed by

\begin{eqnarray}    \label{E:4.1}
a_{0}(y) & = & \frac{\partial}{\partial s}\bigg|_{s=0} \left( \frac{1}{(2 \pi)^{m-1}} \int_{T_{y}^{\ast}Y}
\frac{1}{2 \pi i} \int_{\gamma} \mu^{-s} \Tr \widetilde{r}_{-m} \left(y, \xi, \frac{\lambda}{|\lambda|}, \mu\right) d\mu d \xi \right).
\end{eqnarray}

\noindent
Let $\nabla^{Y}$ be the Levi-Civita connection on $Y$ associated to the induced metric from $g$.
We denote by $R_{\alpha\beta\gamma\delta}$ and $\Ric_{\alpha\beta}$
the Riemann curvature tensor and Ricci tensor on $Y$ associated to $\nabla^{Y}$ defined by

\begin{eqnarray}     \label{E:4.100}
R_{\alpha\beta\gamma\delta} & = & \big\langle \nabla^{Y}_{\partial_{x_{\alpha}}} \nabla^{Y}_{\partial_{x_{\beta}}} \partial_{x_{\gamma}} -
\nabla^{Y}_{\partial_{x_{\beta}}} \nabla^{Y}_{\partial_{x_{\alpha}}} \partial_{x_{\gamma}} -
\nabla^{Y}_{[\partial_{x_{\alpha}}, \partial_{x_{\beta}}]} \partial_{x_{\gamma}}, ~ \partial_{x_{\delta}} \big\rangle_{Y},   \quad
\Ric_{\alpha\beta} = \sum_{\gamma=1}^{m-1} R_{\alpha\gamma\gamma\beta}.
\end{eqnarray}

\noindent
The following lemma is shown in \cite{PS} and \cite{Vi}.

\begin{lemma}   \label{Lemma:4.1}
We consider the boundary normal coordinate system on an open neighborhood $U_{\epsilon_{0}}$ of $y_{0} \in Y$ with metric tensor $g = (g_{ij})$
and $y_{0} = (0, \cdots, 0)$. Then, we have the following equalities:
\begin{eqnarray*}
& (1) & g^{\alpha\beta;\alpha\beta}(y_{0}) = - ~ \frac{1}{3} R_{\alpha\beta\beta\alpha}(y_{0}), \qquad
 g^{\alpha\alpha; \beta\beta}(y_{0}) = ~ \frac{2}{3} R_{\alpha\beta\beta\alpha}(y_{0}),  \\
& (2) & \partial_{y_{\alpha}} \partial_{y_{\alpha}} \ln \big| g \big|(y_{0}) = - \frac{2}{3} \Ric_{\alpha\alpha}(y_{0}), \qquad
\tau_{Y}(y_{0}) = \sum_{\alpha, \beta=1}^{m-1} R_{\alpha\beta\beta\alpha}(y_{0}) = - \sum_{\alpha, \gamma=1}^{m-1} R_{\alpha\beta\alpha\beta}(y_{0}), \\
& (3) & g^{\alpha\beta;m}(y_{0}) = 2 \kappa_{\alpha} \delta_{\alpha\beta} = - g_{\alpha\beta;m}(y_{0}),  \quad
 \int_{{\mathbb R}^{m-1}} |\xi|^{k} \sum_{\alpha ,\beta ,\gamma ,\epsilon =1}^{m-1} g^{\alpha\beta;\gamma\epsilon} \xi_{\alpha} \xi_{\beta} \xi_{\gamma} \xi_{\epsilon} d\xi ~ = ~ 0 ~~ \text{for} ~~ k < -3-m, \\
& (4) & \sum_{\alpha=1}^{m-1} g^{\alpha\alpha;m m}(y_{0}) = 8 \sum_{\alpha=1}^{m-1} \kappa_{\alpha}^{2}(y_{0}) -
\sum_{\alpha=1}^{m-1} g_{\alpha\alpha;m m}(y_{0}),  \\
& (5) & \sum_{\alpha=1}^{m-1} g_{\alpha\alpha;m m}(y_{0}) = - ~ \left\{ \tau_{M}(y_{0}) - \tau_{Y}(y_{0}) - 2 (m-1)^{2} H_{1}^{2}(y_{0}) + 3 (m-1)(m-2) H_{2}(y_{0}) \right\},
\end{eqnarray*}
where $\tau_{M}(y_{0})$ and $\tau_{Y}(y_{0})$ are scalar curvatures of $M$ and $Y$ at $y_{0} \in Y$, respectively, and $H_{1}$ and $H_{2}$ are defined in (\ref{E:2.32}).
\end{lemma}

\noindent
The following lemma is straightforward.

\begin{lemma}   \label{Lemma:4.2}
Let ${\mathbb C}^+ = {\mathbb C} -\{ r\in {\mathbb R} \mid r\leq 0\}$. For $z\in {\mathbb C}^+$ let $\gamma$ be a counterclockwise contour in ${\mathbb C}^+$ with
$z$ inside $\gamma$.
Then for $\re s > 2$ the following integrals are all well defined and one computes:
\begin{eqnarray*}
& & \frac{1}{2 \pi i} \int_{\gamma} \frac{\mu^{-s}}{\mu - z} d\mu = z^{-s}, \quad
\frac{1}{2 \pi i} \int_{\gamma} \frac{\mu^{-s}}{(\mu - z)^{2}} d\mu = - s z^{-s-1}, \quad
\frac{1}{2 \pi i} \int_{\gamma} \frac{\mu^{-s}}{(\mu - z)^{3}} d\mu = \frac{1}{2} s (s + 1) z^{-s-2},  \\
& & \frac{1}{4 \pi^{2}} \int_{{\mathbb R}^{2}} ( |\xi|^{2} + 1)^{-\frac{s}{2}} d \xi ~ = ~ \frac{1}{2 \pi} \frac{1}{s-2},  \qquad
 \frac{1}{4 \pi^{2}} \int_{{\mathbb R}^{2}} ( |\xi|^{2} + 1)^{-\frac{s}{2} -1} d \xi ~ = ~ \frac{1}{2 \pi} \frac{1}{s},  \\
& & \frac{1}{4 \pi^{2}} \int_{{\mathbb R}^{2}} \xi_{1}^{2} ( |\xi|^{2} + 1)^{-\frac{s}{2} - 2} d \xi ~ = ~ \frac{1}{2 \pi} \frac{1}{s (s + 2)},  \\
& & \frac{1}{4 \pi^{2}} \int_{{\mathbb R}^{2}} \xi_{1}^{2} \xi_{2}^{2} ( |\xi|^{2} + 1)^{-\frac{s}{2} - 3} d \xi ~ = ~ \frac{1}{2 \pi} \frac{1}{s (s + 2) (s + 4)},  \\
& & \frac{1}{4 \pi^{2}} \int_{{\mathbb R}^{2}} \xi_{1}^{4} ( |\xi|^{2} + 1)^{-\frac{s}{2} - 3} d \xi ~ = ~ \frac{3}{2 \pi} \frac{1}{s (s + 2)(s + 4)}.
\end{eqnarray*}
\end{lemma}

Now we proceed the computation as in \cite{KL4}.
For two integrable functions $f(\xi)$ and $g(\xi)$ on ${\mathbb R}^{m-1}$, we define an equivalence relation
"$~ \approx ~$" as follows:

\begin{eqnarray}   \label{E:4.2}
f & \approx & g  \qquad \text{if and only if} \qquad \int_{{\mathbb R}^{m-1}} f(\xi) ~ d \xi ~ = ~ \int_{{\mathbb R}^{m-1}} g(\xi) ~ d \xi.
\end{eqnarray}

We first suppose that $Y$ is a $1$-dimensional manifold, {\it i.e.} $m = 2$. Using (\ref{E:3.1}),
we have, at $(y, 0) \in Y$,

\begin{eqnarray}   \label{E:4.3}
\widetilde{r}_{-2} & = &  \left( \mu - \sqrt{|\xi|^2 + \lambda} \right)^{-1}
\bigg\{ \partial_{\xi} \widetilde{\alpha}_{1} \cdot D_{y} \widetilde{r}_{-1} + \widetilde{\alpha}_{0} \cdot
\widetilde{r}_{-1} \bigg\} ~ \approx ~
 \left( \mu - \sqrt{|\xi|^2 + \lambda} \right)^{-2} \cdot \widetilde{\alpha}_{0}  \\
& = & \frac{1}{( \mu - \sqrt{|\xi|^2 + \lambda})^{2}}  \bigg\{ \frac{- \partial_{\xi} \widetilde{\alpha}_{1} \cdot D_{y} \widetilde{\alpha}_{1} + \widetilde{p}_{1}}{2 \sqrt{|\xi|^2 + \lambda}} +
( \widetilde{\omega}_{m} - \frac{1}{2} \widetilde{A}(y, 0)) + \frac{\partial_{y_{m}} \widetilde{\alpha}_{1}}{2 \sqrt{|\xi|^2 + \lambda}} \bigg\}   \nonumber  \\
& \approx & \frac{1}{( \mu - \sqrt{|\xi|^2 + \lambda})^{2}} \bigg\{
( \widetilde{\omega}_{m} - \frac{1}{2} \widetilde{A}(y, 0) + \frac{\partial_{y_{m}} \widetilde{\alpha}_{1}}{2 \sqrt{|\xi|^2 + \lambda}} \bigg\}  \nonumber \\
& \approx & \frac{1}{( \mu - \sqrt{|\xi|^2 + \lambda})^{2}}
 \bigg\{ \widetilde{\omega}_{m} - \frac{\kappa}{2} \cdot \frac{\lambda}{|\xi|^2 + \lambda} ~ \widetilde{\Id} \bigg\} , \nonumber
\end{eqnarray}

\noindent
where $\kappa(y)$ is the principal curvature on $y \in Y$.
Hence,

\begin{eqnarray}   \label{E:4.4}
&  &  \frac{1}{2 \pi} \int_{T_{y}^{\ast}Y}
\frac{1}{2 \pi i} \int_{\gamma} \mu^{-s} \Tr \widetilde{r}_{-2} \left(y, \xi, \frac{\lambda}{|\lambda|}, \mu\right) d\mu d \xi  \\
& = & \frac{1}{2 \pi} \int_{-\infty}^{\infty} (-s) \Tr \widetilde{\omega}_{m} \frac{1}{\sqrt{\xi^{2} + 1}^{s+1}} d \xi
+ s  \cdot \frac{\kappa}{2} \cdot \frac{1}{2 \pi} \int_{-\infty}^{\infty} \frac{1}{\sqrt{\xi^{2} + 1}^{s+3}} d \xi  \nonumber  \\
& = & - s \cdot \Tr \widetilde{\omega}_{m} \cdot \frac{1}{\pi} \int_{0}^{\infty}  \frac{1}{\sqrt{\xi^{2} + 1}^{s+1}} d \xi ~ + ~ s \cdot \left( \frac{\kappa}{2 \pi} + O(s) \right).   \nonumber
\end{eqnarray}

\noindent
Setting $\xi^{2} = t$ and using the identity
$\int_{0}^{\infty} \frac{t^{a-1}}{(1 + t)^{a+b}} dt = \frac{\Gamma(a) \Gamma(b)}{\Gamma(a+b)}$
\cite{AAR,MOS}, we obtain

\begin{eqnarray}    \label{E:4.5}
 \frac{1}{\pi} \int_{0}^{\infty}  \frac{1}{\sqrt{\xi^{2} + 1}^{s+1}} d \xi & = &
\frac{1}{2 \pi} \int_{0}^{\infty} \frac{t^{-\frac{1}{2}}}{(t+1)^{\frac{s+1}{2}}} dt
~ = ~ \frac{1}{s \cdot \pi}
\frac{\Gamma \left(\frac{1}{2}\right) \Gamma \left( \frac{s}{2} + 1 \right)}{\Gamma \left( \frac{s+1}{2} \right)},
 \nonumber
\end{eqnarray}

\noindent
which leads to

\begin{eqnarray}    \label{E:4.6}
& & \frac{1}{2 \pi} \int_{T_{y}^{\ast}Y}
\frac{1}{2 \pi i} \int_{\gamma} \mu^{-s} \Tr \widetilde{r}_{-2} \left(y, \xi, \frac{\lambda}{|\lambda|}, \mu\right) d\mu d \xi    \\
& = &  - \frac{1}{\pi} \cdot \Tr \widetilde{\omega}_{m} \cdot \frac{\Gamma \left(\frac{1}{2}\right) \Gamma \left( \frac{s}{2} + 1 \right)}{\Gamma \left( \frac{s+1}{2} \right)}  ~ + ~ s \cdot \left( \frac{\kappa}{2 \pi} + O(s) \right).
\nonumber
\end{eqnarray}

\noindent
Taking the derivative with respect to $s$ gives

\begin{eqnarray}   \label{E:4.7}
 a_{0}(y) & = &  - \frac{1}{\pi} \Tr ( \widetilde{\omega}_{m} ) \ln 2 +  \frac{\kappa(y)}{2\pi}.
\end{eqnarray}

\noindent
Lemma \ref{Lemma:3.2} shows that if $q = 0$ then $\widetilde{\omega}_{m} = 0$ and if $q = 1$ then $\widetilde{\omega}_{m} = \kappa(y)$. This leads to the following result.

\begin{theorem}   \label{Theorem:4.3}
When $\Dim Y = 1$, the constant $a_{0}$ in Theorem \ref{Theorem:2.4} is given as follow.
\begin{eqnarray*}
a_{0} & = & \begin{cases} \frac{1}{2 \pi} \int_{Y} \kappa(y) ~ dy &  \quad \text{for} \quad q = 0 \\
\frac{1}{2\pi} (1 - 2 \ln 2) \int_{Y} \kappa(y) ~ dy & \quad \text{for} \quad q=1 . \end{cases}
\end{eqnarray*}
\end{theorem}

\vspace{0.2 cm}
\noindent
{\it Remark} : It follows from (\ref{E:2.10}) that

\begin{eqnarray}   \label{E:4.8}
q_{1} & = & \begin{cases} 0 & \quad \text{for} \quad q = 0 \\  - \frac{1}{2 \pi} \int_{Y} \kappa(y) ~ dy
& \quad \text{for} \quad q=1 , \end{cases}
\end{eqnarray}

\noindent
which agrees with ${\frak a}_{2} - {\frak b}_{2}$ in Corollary \ref{Corollary:2.9} and Corollary \ref{Corollary:2.10}.

\vspace{0.3 cm}

We next consider the case that $Y$ is a $2$-dimensional compact Riemannian manifold, {\it i.e.} $m = 3$. We refer to \cite{KL4} for details. Before computing $a_{0}(y)$, we first consider $a_{1}(y)$, which is by Lemma \ref{Lemma:2.3}

\begin{eqnarray}    \label{E:4.9}
a_{1}(y) & = & \left( {\frak a}_{1}(y) - {\frak b}_{1}(y) \right) - \pi_{0}(y).
\end{eqnarray}

\noindent
Simple computation shows that for ${\frak r}_{0} = \binom{2}{q}$,

\begin{eqnarray}    \label{E:4.10}
\pi_{0}(y) & = & - \frac{\partial}{\partial s}\bigg|_{s=0} \left( \frac{1}{(2 \pi)^{2}} \int_{T_{y}^{\ast}Y}
\frac{1}{2 \pi i} \int_{\gamma} \mu^{-s} \Tr \widetilde{r}_{-1} \left(y, \xi, \frac{\lambda}{|\lambda|}, \mu\right) ~ d\mu d \xi \right)  \\
& = & - {\frak r}_{0} ~ \frac{\partial}{\partial s}\bigg|_{s=0} \left( \frac{1}{(2 \pi)^{2}} \int_{T_{y}^{\ast}Y}
\frac{1}{2 \pi i} \int_{\gamma} \frac{\mu^{-s}}{\mu - \sqrt{|\xi|^{2} + 1}} ~ d\mu d \xi \right)  \nonumber \\
& = & \frac{{\frak r}_{0}}{8 \pi}.  \nonumber
\end{eqnarray}

\noindent
It is well known (for example, Theorem 3.4.1 and Theorem 3.6.1 in \cite{Gi3} or Section 4.2 and 4.5 in \cite{Ki}) that

\begin{eqnarray}    \label{E:4.11}
{\frak a}_{1}(y) - {\frak b}_{1}(y) & = & \frac{{\frak r}_{0}}{8 \pi},
\end{eqnarray}

\noindent
which yields the following result.

\begin{lemma}   \label{Lemma:4.4}
When $\Dim Y = 2$, the constant $a_{1}$ in Lemma \ref{Lemma:2.2} is zero.
\end{lemma}

We next compute $a_{0}(x)$. We recall that

\begin{eqnarray}   \label{E:4.12}
a_{0}(y) & = & \frac{\partial}{\partial s}\bigg|_{s=0} \left( \frac{1}{(2 \pi)^{2}} \int_{T_{y}^{\ast}Y}
\frac{1}{2 \pi i} \int_{\gamma} \mu^{-s} \Tr \widetilde{r}_{-3} \left(y, \xi, \frac{\lambda}{|\lambda|}, \mu\right) d\mu d \xi \right)  \\
& = & \frac{\partial}{\partial s}\bigg|_{s=0} \left( \frac{1}{(2 \pi)^{2}} \int_{T_{y}^{\ast}Y}
\frac{1}{2 \pi i} \int_{\gamma} \mu^{-s} \Tr \left\{ (\I) + (\II) + (\III) + (\IV) + (\V)  \right\} d\mu d \xi \right),   \nonumber
\end{eqnarray}

\noindent
where

\begin{eqnarray*}
& & (\I) ~ = ~ \frac{\sum_{|\omega|=2} \frac{1}{\omega !} \partial_{\xi}^{\omega} \widetilde{\alpha}_{1} \cdot D_{y}^{\omega} \widetilde{r}_{-1}}{\mu - \sqrt{|\xi|^2 + \lambda}}, \qquad
(\II) ~ = ~ \frac{\partial_{\xi} \widetilde{\alpha}_{1} \cdot D_{y} \widetilde{r}_{-2}}{\mu - \sqrt{|\xi|^2 + \lambda}},  \qquad
(\III) ~ = ~ \frac{\partial_{\xi} \widetilde{\alpha}_{0} \cdot D_{y} \widetilde{r}_{-1}}{\mu - \sqrt{|\xi|^2 + \lambda}},  \\
& & (\IV) ~ = ~ \frac{\widetilde{\alpha}_{0} \cdot \widetilde{r}_{-2}}{\mu - \sqrt{|\xi|^2 + \lambda}},  \qquad (\V) ~ = ~ \frac{\widetilde{\alpha}_{-1} \cdot \widetilde{r}_{-1}}{\mu - \sqrt{|\xi|^2 + \lambda}}.
\end{eqnarray*}

\noindent
Moreover, we denote

\begin{eqnarray*}
(\V) ~ = ~ \frac{\widetilde{\alpha}_{-1} \cdot \widetilde{r}_{-1}}{\mu - \sqrt{|\xi|^2 + \lambda}} ~ = ~
\frac{\widetilde{\alpha}_{-1}}{(\mu - \sqrt{|\xi|^2 + \lambda})^{2}} ~ = ~ (\V_{1}) + (\V_{2}) + (\V_{3}) + (\V_{4}) + (\V_{5}) + (\V_{6}) + (\V_{7}) + (\V_{8}),
\end{eqnarray*}

\noindent
where

\begin{eqnarray*}
& & (\V_{1}) ~ = ~ \frac{\sum_{|\omega|=2} \frac{1}{\omega !} \partial_{\xi}^{\omega} \widetilde{\alpha}_{1} \cdot \partial_{y}^{\omega} \widetilde{\alpha}_{1}}{2(\mu - \sqrt{|\xi|^2 + \lambda})^{2} \sqrt{|\xi|^2 + \lambda}}, \qquad
(\V_{2}) ~ = ~ \frac{i (\partial_{\xi} \widetilde{\alpha}_{0}) \cdot (\partial_{y} \widetilde{\alpha}_{1})}{2(\mu - \sqrt{|\xi|^2 + \lambda})^{2} \sqrt{|\xi|^2 + \lambda}},  \\
& & (\V_{3}) ~ = ~ \frac{i (\partial_{\xi} \widetilde{\alpha}_{1}) \cdot (\partial_{y} \widetilde{\alpha}_{0})}{2(\mu - \sqrt{|\xi|^2 + \lambda})^{2} \sqrt{|\xi|^2 + \lambda}},  \qquad
(\V_{4}) ~ = ~ \frac{- ( \I,~ \Oo) ~ \alpha_{0}^{2} ~ ( \I,~ \Oo)^{T}}{2(\mu - \sqrt{|\xi|^2 + \lambda})^{2} \sqrt{|\xi|^2 + \lambda}},  \\
& & (\V_{5}) ~ = ~ \frac{\widetilde{p}_{0}}{2(\mu - \sqrt{|\xi|^2 + \lambda})^{2} \sqrt{|\xi|^2 + \lambda}},  \qquad
(\V_{6}) ~ = ~ \frac{(2 \widetilde{\omega}_{m} - \widetilde{A}(y, 0)) \widetilde{\alpha}_{0}}{2(\mu - \sqrt{|\xi|^2 + \lambda})^{2} \sqrt{|\xi|^2 + \lambda}},  \\
& & (\V_{7}) ~ = ~ \frac{\partial_{y_{m}} \widetilde{\alpha}_{0}}{2(\mu - \sqrt{|\xi|^2 + \lambda})^{2} \sqrt{|\xi|^2 + \lambda}},  \qquad
(\V_{8}) ~ = ~ \frac{ - (\partial_{y_{m}} \widetilde{\omega}_{m} + \widetilde{\omega}_{m} \widetilde{\omega}_{m} - \widetilde{A}(y, 0) \widetilde{\omega}_{m})}{2(\mu - \sqrt{|\xi|^2 + \lambda})^{2} \sqrt{|\xi|^2 + \lambda}}.
\end{eqnarray*}

\noindent
Direct and tedious computations show the followings (cf. \cite{KL4}).
Here, as before, we denote ${\frak r}_{0} = \binom{2}{q}$ so that ${\frak r}_{0} = 1$ for $q = 0,~~ 2$ and ${\frak r}_{0} = 2$ for $q = 1$.

\begin{eqnarray*}
& & \frac{1}{(2 \pi)^{2}} \int_{T_{y}^{\ast}Y} \frac{1}{2 \pi i} \int_{\gamma} \mu^{-s} (\I) d\mu d \xi ~ = ~ - {\frak r}_{0} \cdot \frac{\tau_{Y}}{24 \pi} \cdot \frac{s+1}{s+2}, \\
& & \frac{1}{(2 \pi)^{2}} \int_{T_{y}^{\ast}Y} \frac{1}{2 \pi i} \int_{\gamma} \mu^{-s} (\II) d\mu d \xi ~ = ~
{\frak r}_{0} \cdot \frac{\tau_{Y}}{12 \pi} \cdot \frac{s+1}{s+2}
~ - ~ \frac{1}{4 \pi} \Tr \left( \partial_{y_{\alpha}} \widetilde{\omega}_{\alpha} \right) \cdot \frac{s+1}{s+2}, \\
& & \frac{1}{(2 \pi)^{2}} \int_{T_{y}^{\ast}Y} \frac{1}{2 \pi i} \int_{\gamma} \mu^{-s} (\III) d\mu d \xi ~ = ~ 0, \\
& & \frac{1}{(2 \pi)^{2}} \int_{T_{y}^{\ast}Y} \frac{1}{2 \pi i} \int_{\gamma} \mu^{-s} (\IV) d\mu d \xi ~ = ~  {\frak r}_{0} \cdot \frac{H_{1}^{2}}{4 \pi} \cdot \frac{(s + 1)^{2}(s+3)}{(s+2)(s+4)}
~ - ~ {\frak r}_{0} \cdot \frac{H_{2}}{4 \pi} \cdot \frac{s+1}{(s+2)(s+4)}  \\
& & \hspace{5.5 cm} - \frac{H_{1}}{2\pi} \Tr(\widetilde{\omega}_{m}) \frac{(s+1)^{2}}{s+2} + \frac{1}{4\pi} \Tr (\widetilde{\omega}_{m} \widetilde{\omega}_{m}) \cdot (s+1), \\
& & \frac{1}{(2 \pi)^{2}} \int_{T_{y}^{\ast}Y} \frac{1}{2 \pi i} \int_{\gamma} \mu^{-s} (\V_{1}) d\mu d \xi ~ = ~ - {\frak r}_{0} \cdot \frac{\tau_{Y}}{24 \pi} \cdot \frac{1}{s+2}, \\
& & \frac{1}{(2 \pi)^{2}} \int_{T_{y}^{\ast}Y} \frac{1}{2 \pi i} \int_{\gamma} \mu^{-s} (\V_{2}) d\mu d \xi ~ = ~ 0, \\
& & \frac{1}{(2 \pi)^{2}} \int_{T_{y}^{\ast}Y} \frac{1}{2 \pi i} \int_{\gamma} \mu^{-s} (\V_{3}) d\mu d \xi ~ = ~ {\frak r}_{0} \cdot \frac{\tau_{Y}}{12 \pi} \cdot \frac{1}{s+2}
~ - ~ \frac{1}{4 \pi} \Tr \left( \partial_{y_{\alpha}} \widetilde{\omega}_{\alpha} \right) \cdot \frac{1}{s+2}, \\
& & \frac{1}{(2 \pi)^{2}} \int_{T_{y}^{\ast}Y} \frac{1}{2 \pi i} \int_{\gamma} \mu^{-s} (\V_{4}) d\mu d \xi ~ = ~
{\frak r}_{0} \cdot \frac{H_{1}^{2}}{4 \pi} \cdot \frac{(s+1)(s+3)}{(s + 2)(s+4)}
- {\frak r}_{0} \cdot \frac{H_{2}}{4 \pi} \cdot \frac{1}{(s+2)(s+4)}  \\
& & \hspace{5.5 cm} -  \frac{1}{4 \pi} \Tr (\widetilde{\omega_{\alpha} \omega_{\alpha}}) \cdot \frac{1}{s+2}
- \frac{H_{1}}{2\pi} \Tr ( \widetilde{\omega}_{m}) \cdot \frac{s+1}{s+2} + \frac{1}{4\pi} \Tr ( \widetilde{\omega}_{m} \widetilde{\omega}_{m}), \\
& & \frac{1}{(2 \pi)^{2}} \int_{T_{y}^{\ast}Y} \frac{1}{2 \pi i} \int_{\gamma} \mu^{-s} (\V_{5}) d\mu d \xi ~ = ~
\frac{1}{4 \pi} \Tr \left( \partial_{y_{\alpha}} \widetilde{\omega}_{\alpha} \right)
~ + ~ \frac{1}{4 \pi} \Tr \left( \widetilde{\omega_{\alpha} \omega_{\alpha}} \right)  +
\frac{1}{4 \pi} \Tr \big( \widetilde{E}_{q} \big),
\end{eqnarray*}

\begin{eqnarray*}
& & \frac{1}{(2 \pi)^{2}} \int_{T_{y}^{\ast}Y} \frac{1}{2 \pi i} \int_{\gamma} \mu^{-s} (\V_{6}) d\mu d \xi ~ = ~ - {\frak r}_{0} \cdot \frac{H_{1}^{2}}{2 \pi} \cdot \frac{s+1}{s+2} + \frac{H_{1}}{2\pi} \Tr (\widetilde{\omega}_{m}) \cdot \frac{2s+3}{s+2} - \frac{1}{2\pi}
\Tr ( \widetilde{\omega}_{m} \widetilde{\omega}_{m}), \\
& & \frac{1}{(2 \pi)^{2}} \int_{T_{y}^{\ast}Y} \frac{1}{2 \pi i} \int_{\gamma} \mu^{-s} (\V_{7}) d\mu d \xi ~ = ~ {\frak r}_{0} \cdot \frac{1}{16 \pi} (\tau_{M} - \tau_{Y}) \cdot \frac{s+1}{s+2}
+ {\frak r}_{0} \cdot \frac{H_{1}^{2}}{2 \pi} \cdot \frac{s+1}{s+4}  \\
& & \hspace{5.5 cm} - ~  {\frak r}_{0} \cdot \frac{H_{2}}{8 \pi} \cdot \frac{s^{2} + s - 4}{(s+2)(s+4)}
- \frac{1}{4\pi} \Tr (\partial_{y_{m}} \widetilde{\omega}_{m}), \\
& & \frac{1}{(2 \pi)^{2}} \int_{T_{y}^{\ast}Y} \frac{1}{2 \pi i} \int_{\gamma} \mu^{-s} (\V_{8}) d\mu d \xi ~ = ~ - \frac{H_{1}}{2\pi} \Tr(\widetilde{\omega}_{m}) + \frac{1}{4 \pi} \Tr (\partial_{y_{m}} \widetilde{\omega}_{m}) +
\frac{1}{4 \pi} \Tr (\widetilde{\omega}_{m} \widetilde{\omega}_{m}).
\end{eqnarray*}

\noindent
Adding up the above terms, we obtain

\begin{eqnarray}    \label{E:4.13}
& & \frac{1}{(2 \pi)^{2}} \int_{T_{y}^{\ast}Y}
\frac{1}{2 \pi i} \int_{\gamma} \mu^{-s} \Tr \widetilde r_{-3} \left(y, \xi, \frac{\lambda}{|\lambda|}, \mu\right) d\mu d \xi \\
& = & {\frak r}_{0} \cdot \left\{ \frac{\tau_{M}}{16 \pi} \cdot \frac{s+1}{s+2}  -  \frac{\tau_{Y}}{48 \pi} \cdot \frac{s-1}{s+2}  +  \frac{H_{1}^{2}}{4 \pi} \cdot \frac{s^{3} + 6s^{2} + 7s + 2}{(s+2)(s+4)}  -
\frac{H_{2}}{8 \pi} \cdot \frac{s(s+3)}{(s+2)(s+4)}   \right\}  +  \frac{1}{4 \pi} \Tr (\widetilde{E}_{q})  \nonumber \\
&  & + ~ \frac{1}{4 \pi} \Tr(\widetilde{\omega_{\alpha} \omega_{\alpha}}) \cdot \frac{s+1}{s+2}
- \frac{H_{1}}{2 \pi} \Tr(\widetilde{\omega}_{m}) \cdot \frac{s^{2} + 2s + 1}{s+2} +
\frac{1}{4 \pi} \Tr(\widetilde{\omega}_{m} \widetilde{\omega}_{m}) \cdot (s+1),  \nonumber
\end{eqnarray}

\noindent
which shows that

\begin{eqnarray}    \label{E:4.14}
a_{0}(y) & = & {\frak r}_{0} \cdot \left( \frac{\tau_{M}}{64 \pi} - \frac{\tau_{Y}}{64 \pi} + \frac{11}{64 \pi} H_{1}^{2}
- \frac{3}{64 \pi} H_{2} \right)  \\
& &  + \frac{1}{16 \pi} \sum_{\alpha=1}^{2} \Tr(\widetilde{\omega_{\alpha} \omega_{\alpha}})
- \frac{3H_{1}}{8 \pi} \Tr (\widetilde{\omega}_{m}) + \frac{1}{4 \pi} \Tr(\widetilde{\omega}_{m} \widetilde{\omega}_{m}).  \nonumber
\end{eqnarray}

\noindent
If $p = 0$, then ${\frak r}_{0} = 1$ and $\widetilde{\omega}_{\alpha} = \widetilde{\omega}_{m} = 0$.
Eq.(\ref{E:3.27}) and (\ref{E:3.28}) show that if $p=1$, then ${\frak r}_{0} = 2$ and

\begin{eqnarray*}
\widetilde{\omega}_{m} ~ = ~ \left( \begin{array}{clcr} \kappa_{1} & 0 \\ 0 & \kappa_{2} \end{array} \right), \quad
\widetilde{\omega}_{m} \widetilde{\omega}_{m} ~ = ~ \left( \begin{array}{clcr} \kappa_{1}^{2} & 0 \\ 0 & \kappa_{2}^{2} \end{array} \right), \quad
\widetilde{\omega_{1} \omega_{1}} ~ = ~ - \kappa_{1}^{2} \left( \begin{array}{clcr} 1 & 0 \\ 0 & 0 \end{array} \right), \quad
\widetilde{\omega_{2} \omega_{2}} ~ = ~ - \kappa_{2}^{2} \left( \begin{array}{clcr} 0 & 0 \\ 0 & 1 \end{array} \right).
\end{eqnarray*}

\noindent
If $p = 2$, then ${\frak r}_{0} = 1$ and

\begin{eqnarray*}
\widetilde{\omega}_{m} ~ = ~ \kappa_{1} + \kappa_{2} = 2 H_{1}, \quad
\widetilde{\omega}_{m} \widetilde{\omega}_{m} ~ = ~ (\kappa_{1} + \kappa_{2})^{2} = 4H_{1}^{2}, \quad
\widetilde{\omega_{1} \omega_{1}} ~ = ~ - \kappa_{1}^{2}, \quad
\widetilde{\omega_{2} \omega_{2}} ~ = ~ - \kappa_{2}^{2}.
\end{eqnarray*}

\noindent
These facts lead to the following result.

\begin{theorem}  \label{Theorem:4.5}
When $\Dim Y = 2$, the constant $a_{0}$ and $a_{1}$ in Theorem \ref{Theorem:2.4} and Lemma \ref{Lemma:2.2} are given as follows.

\begin{eqnarray*}
a_{1} & = & 0,  \\
a_{0} & = & \begin{cases} \frac{1}{64 \pi} \int_{Y} \left( \tau_{M} - \tau_{Y} + 11 H_{1}^{2} - 3 H_{2} \right) ~ dy & \text{for} \quad q = 0  \\
\frac{1}{32 \pi} \int_{Y} \left( \tau_{M} - \tau_{Y} + 11 H_{1}^{2} - 15 H_{2} \right) ~ dy  & \text{for} \quad q = 1  \\
\frac{1}{64 \pi} \int_{Y} \left( \tau_{M} - \tau_{Y} + 11 H_{1}^{2} + 5 H_{2} \right) ~ dy  & \text{for} \quad q = 2 .
\end{cases}
\end{eqnarray*}
\end{theorem}

\noindent
{\it Remark} : When $\Dim Y = 2$, we get from (\ref{E:2.10}), (\ref{E:3.78}), (\ref{E:4.13})  and Lemma \ref{Lemma:2.3}

\begin{eqnarray}    \label{E:4.15}
q_{2} & = & \int_{Y} q_{2}(y) ~ dy ~ = ~ \frac{1}{2} \int_{Y} \Big\{ {\frak r}_{0} \cdot \left( \frac{\tau_{M}}{32 \pi} ~ + ~ \frac{\tau_{Y}}{96 \pi} ~ + ~ \frac{H_{1}^{2}}{16 \pi} \right) ~ + ~  \frac{1}{4 \pi} \Tr (\widetilde{E}_{q}) \\
& &  + ~  \frac{1}{8 \pi} \sum_{\alpha=1}^{2} \Tr (\widetilde{\omega_{\alpha} \omega_{\alpha}}) ~ - ~  \frac{H_{1}}{4 \pi} \Tr (\widetilde{\omega}_{m})
~ + ~  \frac{1}{4 \pi} \Tr (\widetilde{\omega}_{m} \widetilde{\omega}_{m}) \Big\} ~ dy    \nonumber \\
& = & \begin{cases} \frac{1}{8 \pi} \int_{Y} \left( \frac{1}{8} \tau_{M} + \frac{1}{24} \tau_{Y} + \frac{1}{4} H_{1}^{2} ~ \right) ~ dy & \text{for} \quad q = 0  \\
\frac{1}{8 \pi} \int_{Y} \left( - \frac{1}{4} \tau_{M} - \frac{5}{12} \tau_{Y} + \frac{1}{2} H_{1}^{2}  \right) ~ dy  & \text{for} \quad q = 1  \\
\frac{1}{8 \pi} \int_{Y} \left( - \frac{3}{8} \tau_{M} + \frac{13}{24} \tau_{Y} + \frac{1}{4}  H_{1}^{2} \right) ~ dy  & \text{for} \quad q = 2   \end{cases}   \nonumber \\
& = & {\frak b}_{3} - {\frak c}_{3} ~ = ~ \frac{1}{2} \left( \zeta_{Q^{q}_{\Abs}(0)}(0) + \ell_{q} \right),   \nonumber
\end{eqnarray}

\noindent
which agrees with Corollary \ref{Corollary:2.10}.

\vspace{0.3 cm}

As an application of Corollary \ref{Corollary:2.5} and Theorem \ref{Theorem:4.3}, we recover Theorem 1.1 in \cite{GG}.
For this purpose let $M$ be a $2$-dimensional compact Riemann manifold with boundary $Y$.
We consider a Laplacian $\Delta^{0}_{M}$ acting on smooth functions and the conformal variation of Corollary \ref{Corollary:2.5} as follows. For a smooth function $F : M \rightarrow {\mathbb R}$, we denote $g_{ij}(\epsilon) = e^{2 \epsilon F} g_{ij}$.
We also denote by $ \ln \Det \Delta^{0}_{M, \Abs}(\epsilon)$, $\ln \Det \Delta^{0}_{M, \Dir}(\epsilon)$, $\kappa(\epsilon)$, $dy(\epsilon)$, and $~ \ln \Det Q^{0}_{\Abs}(0)(\epsilon) $
the corresponding objects with respect to the metric $g_{ij}(\epsilon)$, where $\Delta^{0}_{M}(\epsilon) = e^{-2 \epsilon F} \Delta^{0}_{M}$ and $Q^{0}_{\Abs}(0)(\epsilon) = e^{- \epsilon F} Q^{0}_{\Abs}(0)$.
Then, Corollary \ref{Corollary:2.5} and Theorem \ref{Theorem:4.3} for $g_{ij}(\epsilon)$ can be rewritten by

\begin{eqnarray}    \label{E:4.16}
& & \ln \frac{\Det^{\ast} Q^{0}_{\Abs}(0)(\epsilon)}{\ell(Y)(\epsilon)} ~ = ~ - \frac{1}{2 \pi} \int_{Y} \kappa(\epsilon) dx(\epsilon) ~ - ~ \ln V(M)(\epsilon) ~ + ~
\ln \Det^{\ast} \Delta^{0}_{M, \Abs}(\epsilon) - \ln \Det \Delta^{0}_{M, \Dir}(\epsilon)    \nonumber \\
&  & \hspace{1.0 cm} = ~ - \frac{1}{2 \pi} \int_{Y} \kappa(\epsilon) e^{\epsilon F} ~ dy ~ - ~ \ln \int_{M} e^{2 \epsilon F} ~ dx ~ + ~
\ln \Det^{\ast} \Delta^{0}_{M, \Abs}(\epsilon) - \ln \Det \Delta^{0}_{M, \Dir}(\epsilon),
\end{eqnarray}

\noindent
where $dx = d \vol(M)$.
For $t \rightarrow 0^{+}$, we put

\begin{eqnarray}    \label{E:4.17}
\Tr \left( F e^{-t \Delta^{0}_{M, \Abs}} \right) & \sim & \sum_{j=0}^{\infty} {\frak a}_{j}(F) t^{- \frac{2-j}{2}}, \qquad
\Tr \left( F e^{-t \Delta^{0}_{M, \Dir}} \right) ~ \sim ~ \sum_{j=0}^{\infty} {\frak b}_{j}(F) t^{- \frac{2-j}{2}}.
\end{eqnarray}

\noindent
It is well known that \cite{BG1,BG2}

\begin{eqnarray}    \label{E:4.18}
& & \frac{d}{d \epsilon}\big|_{\epsilon = 0} \ln \Det^{\ast} \Delta^{0}_{M, \Abs}(\epsilon) ~ = ~ - 2 \left( {\frak a}_{2}(F) - \frac{1}{\vol(M)} \int_{M} F(x) ~ dx \right), \\
& & \frac{d}{d \epsilon}\big|_{\epsilon = 0} \ln \Det^{\ast} \Delta^{0}_{M, \Dir}(\epsilon) ~ = ~ - 2 {\frak b}_{2}(F),  \qquad
\frac{d}{d \epsilon}\big|_{\epsilon = 0} \kappa(\epsilon) ~ = ~ - F \kappa - F_{;2},   \nonumber
\end{eqnarray}

\noindent
where $F_{;2}$ is the derivative of $F$ with respect to the inward unit normal vector field. Moreover, it is also well known that \cite{Gi3,Ki}
\begin{eqnarray}    \label{E:4.19}
{\frak a}_{2}(F) - {\frak b}_{2}(F) & = & \frac{1}{4 \pi} \int_{Y} F_{;2} ~ dy.
\end{eqnarray}

\noindent
Consideration of all these facts shows that

\begin{eqnarray}     \label{E:4.20}
\frac{d}{d \epsilon}\big|_{\epsilon = 0} \ln \frac{\Det^{\ast} Q(0)(\epsilon)}{\ell(Y)(\epsilon)} & = & 0,
\end{eqnarray}

\noindent
which shows that $\frac{\Det^{\ast} Q^{0}_{\Abs}(0)}{\ell(Y)}$ is a conformal invariant, which is proved earlier in \cite{GG} (see also \cite{EW}).

\vspace{0.2 cm}

\noindent
{\it Remark} : A similar computation shows that $\ln \Det^{\ast} Q^{1}_{\Abs}(0)$ depends on the conformal change of a metric, where $Q^{1}_{\Abs}(0)$ is the Dirichlet-to-Neumann operator acting on $1$-forms on $Y$ with $\Dim Y = 1$.

\vspace{0.5 cm}

\section*{Declarations}
\subsection*{Ethics approval and consent to participate}
Not applicable.
\subsection*{Consent for publication}
Not applicable.
\subsection*{Availability of data and materials}
Data sharing not applicable to this article as no datasets were generated or analysed during the current study.
\subsection*{Competing interests}
On behalf of all authors, the corresponding author states that there is no conflict of interest.
\subsection*{Funding}
The second author was supported by the National Research Foundation of Korea with the Grant number 2016R1D1A1B01008091.
\subsection*{Authors' contributions}
Both authors contributed equally to the manuscript.


\end{document}